\documentclass[11pt]{article}   

\usepackage{graphicx}

\usepackage{latexsym,amsfonts,amsmath,theorem,amssymb}
\usepackage{stmaryrd,array,tabularx,bbm}
\usepackage{pstricks,graphicx}

\usepackage{a4wide}

\usepackage{color}
\usepackage{caption}
\usepackage{theorem}
\usepackage{hyperref}
\usepackage{cleveref}
\usepackage{framed}
\usepackage{blindtext}
\usepackage{float}
\usepackage{todonotes}

\usepackage[caption=false]{subfig}
\usepackage{algorithm}
\usepackage{algorithmicx}
\usepackage[noend]{algpseudocode}
\usepackage{tikz}

\algrenewcommand\algorithmicrequire{\textbf{Input:}}
\algrenewcommand\algorithmicensure{\textbf{Output:}}
\newtheorem{theorem}{Theorem}[section]
\newtheorem{lemma}[theorem]{Lemma}
\newtheorem{proposition}[theorem]{Proposition}

\newtheorem{definition}[theorem]{Definition}

\newenvironment{proof}{\begin{trivlist}
    \item[\hskip\labelsep{\bf Proof.}]}{$\hfill\Box$\end{trivlist}}

{\theoremstyle{plain} \theorembodyfont{\rmfamily}

\newtheorem{remark}[theorem]{Remark}}

\numberwithin{equation}{section}
\numberwithin{figure}{section}
\numberwithin{table}{section}

\newcommand{\bsi}{{\boldsymbol{i}}}

 \newcommand{\hPhi}{{\bar{\Phi}}}
  \newcommand{\rPhi}{{\Phi_{\rm r}}}
   \newcommand{\dPhi}{{\Phi_{\rm d}}}
     \newcommand{\lPhi}{{\Phi_{\ell}}}
   \newcommand{\wPhi}{{\Phi_{\rm w}}}
     \newcommand{\rg}{{{\rm r}}}
   \newcommand{\dt}{{{\rm d}}}
   \newcommand{\lin}{{{\ell}}}
   \newcommand{\dw}{{{\rm w}}}
  
   \newcommand{\AtA}{\boldsymbol{A}}
   \newcommand{\Atb}{\boldsymbol{b}}
      \newcommand{\PtP}{\boldsymbol{P}}
   \newcommand{\Ptd}{\boldsymbol{d}}

\newcommand{\argmin}{{\mathrm{argmin}}}

\renewcommand{\hat}{\widehat}
\newcommand{\vertiii}[1]{{\left\vert\kern-0.25ex\left\vert\kern-0.25ex\left\vert #1 
\right\vert\kern-0.25ex\right\vert\kern-0.25ex\right\vert}}

\newcommand{\review}{}

\title{Gradient flows and randomised thresholding: sparse inversion and classification}
\author{Jonas Latz\thanks{The author thanks the anonymous reviewers for their helpful remarks.}}
\date{\footnotesize Maxwell Institute for Mathematical Sciences and \\ \footnotesize School of Mathematical and Computer Sciences, Heriot-Watt University, Edinburgh, UK\\
\url{j.latz@hw.ac.uk}}

\begin{document}
\maketitle
\begin{abstract}
    Sparse inversion and classification problems are ubiquitous in modern data science and imaging. They are often formulated  as non-smooth minimisation problems. In sparse inversion, we minimise, e.g., the sum of a data fidelity term and an L1/LASSO regulariser. In classification, we consider, e.g., the sum of a data fidelity term and a non-smooth Ginzburg--Landau energy. Standard (sub)gradient descent methods have shown to be inefficient when approaching such problems. Splitting techniques are much more useful: here, the target function is partitioned into a sum of two subtarget functions -- each of which can be efficiently optimised. Splitting proceeds by performing optimisation steps alternately with respect to each of the two subtarget functions. 
    
    In this work, we study splitting from a stochastic continuous-time perspective. Indeed, we define a differential inclusion that follows one of the two subtarget function's negative sub\review{differential} at each point in time. The choice of the subtarget function is controlled by a binary continuous-time Markov process. The resulting dynamical system is a stochastic approximation of the underlying subgradient flow. We investigate this stochastic approximation for an L1-regularised sparse inversion flow and for a discrete Allen-Cahn equation minimising a Ginzburg--Landau energy. In both cases, we study the longtime behaviour of the stochastic dynamical system and its ability to approximate the underlying subgradient flow at any accuracy. We illustrate our theoretical findings in a simple sparse estimation problem and also in  low- \review{and high-dimensional} classification problem\review{s}.
\end{abstract}

\section{Introduction}
Inverse and classification problems arise in a multitude of scientific disciplines. In inverse problems, we aim at estimating parameters of mathematical models through indirect measurements. A usual example is the reconstruction of an image of \review{a} patient's brain structure given measurements from X-rays that \review{are} attenuated by the patient's brain. In classification, our goal is to partition data into two or more different classes. In the previous example, we could aim to partition the pixels in the reconstructed image with respect to the area of the brain they refer to.

Inverse and classification problems can often be phrased as optimisation problems: minimising the sum of a negative loglikelihood -- representing the relationship between data and parameter -- and a regulariser -- forcing the parameter to satisfy certain regularity properties. In image reconstruction and classification, certain non-smooth regularisers have gained popularity.   \review{ In image reconstruction or, generally, sparse inversion}, an $\ell_1$/LASSO-regulariser ensures that the image or object of interest is represented sparsely\review{, e.g., within a dictionary or on a basis}, see \cite{CANDES2011}. In binary classification, a (possibly non-smooth) Ginzburg--Landau functional is used to separate data \review{into two classes} by forcing the values to concentrate close to one of two values, e.g., $-1$ and $1$, see \cite{Bertozzi}. \review{Both, sparse inversion and binary classification are highly relevant in modern data science and imaging, see, e.g. \cite{Ahishakiye,Yang2016-wp,Zhao} and \cite{MIN20095256,Guoshao,Tufail2020-ee}, respectively.}

To understand the solution procedure of these optimisation problems, the associated subgradient flows have been studied extensively, see, e.g.,  \cite{BENES2004,BRUCK1975,Bungert_2019,Marcellin2006}. \review{The subgradient flows are interesting as they give a continuous-time representation of a subgradient descent-type algorithm.} Due to the non-smoothness of the minimisation problem, the associated flows are often non-smooth as well. Thus, when attempting to construct an optimisation algorithm from \review{a} subgradient flow, one has to be very careful: Non-smooth dynamical systems tend to be difficult to solve with traditional numerical techniques. In practice, the non-smoothness is often relaxed or optimisation methods are constructed with more advanced numerical techniques. In this work, we follow the second approach: we propose and discuss randomised splitting techniques in a continuous-time setting. 
\subsection{Problem setting and motivation.} \label{subsec_Problem_Setting}
\paragraph{Stochastic approximations of differential inclusions.} Throughout this work, we study differential inclusions on a space $X:=\mathbb{R}^n$ that are of the form 
$$
\frac{\dt{x}(t)}{\dt t} \in F({x}(t)) + G(x(t)) \qquad (t \geq 0), \qquad \qquad x(0) = x^{(0)} \in X,  
$$
where $F,  G: X \rightarrow 2^{X}$ describe set-valued maps (usually negative sub\review{differentials} of potentials, see \cite{Brezis} for details).
Here, we are particularly interested in cases, where the two \emph{subflows} 
\begin{align*}
    \frac{\dt{x_F}(t)}{\dt t} &\in 2F({x_F}(t))  \qquad (t \geq 0),\qquad \qquad  x_F(0) = x^{(0)}_F \in X, \\
     \frac{\dt{x_G}(t)}{\dt t} &\in  2G(x_G(t)) \qquad (t \geq 0), \qquad \qquad  x_G(0) = x^{(0)}_G \in X,
\end{align*}
are well-understood, i.e. they can be solved analytically or simulated in an efficient, stable, and accurate way. \review{In general, t}here is no \review{exact} way of combining the paths $({x_F}(t))_{t \geq 0}$ and $({x_G}(t))_{t \geq 0}$ in a way that reconstructs $({x}(t))_{t \geq 0}$. In this work, we study the stochastic approximation of $({x}(t))_{t \geq 0}$ that we obtain through following one of the trajectories $({x_F}(t))_{t \geq 0}$ or $({x_G}(t))_{t \geq 0}$ for a random waiting time before switching to the other, which we  then again follow for a random waiting time and so on. \review{This stochastic approximation is indeed a randomised splitting technique.}

We now describe th\review{e} dynamical system more particularly. Let $(\Omega, \mathcal{F}, \mathbb{P})$ be a probability space and let $\bsi: \Omega \times [0, \infty) \rightarrow \{0,1\}$ be a continuous-time Markov process (CTMP) on $\{0,1\}$ that  has transition rate matrix
\begin{equation*}
    \begin{pmatrix}-\lambda & \lambda \\ \lambda & -\lambda \end{pmatrix},
\end{equation*}
for some \emph{switching rate} $\lambda > 0$. That means, $(\bsi(t))_{t \geq 0}$ is a piecewise constant, right-continuous process that stays in a state for an exponentially distributed waiting time $\delta T \sim \mathrm{Exp}(\lambda)$ with mean $\mathbb{E}[\delta T] = 1/\lambda$ \review{before switching to the other state}.

The value $\bsi(t)$ determines which of the flows $({x_F}(t))_{t \geq 0}$ or $({x_G}(t))_{t \geq 0}$ we follow at time $t \geq 0$. Indeed, we consider the following dynamical system
$$
\frac{\dt \hat{x}(t) }{\dt t} \in  \begin{cases} 2F(\hat{x}(t)), \qquad &\text{ if } \bsi(t) = 0, \\
2G(\hat{x}(t)), \qquad &\text{ if } \bsi(t) = 1, 
\end{cases} \qquad \qquad (t \geq 0).
$$
Following the work by Robbins and Monro \cite{RobbinsMonro}, we refer to $(\hat{x}(t))_{t \geq 0}$ as a \emph{stochastic approximation} of $(x(t))_{t \geq 0}$. In a setting where the differential inclusions are actually differential equations and sufficiently smooth, one can sometimes show that $(\hat{x}(t))_{t \geq 0} \rightarrow (x(t))_{t \geq 0}$ in a weak sense, as $\lambda \rightarrow 0$. We refer to \cite{Jin2021,Latz2021} for results of this type and a general perspective on stochastic approximation in continuous time.

\paragraph{Sparse inversion and classification.} \review{Throughout this work, we study the stochastic approximation of two specific subgradient flows that are used to solve two different non-smooth optimisation problems: a \emph{sparse inverse problem} and a \emph{binary classification problem}, respectively.}

\review{In both cases, we minimise the sum of a linear-quadratic data misfit and a certain regulariser. This regulariser is an $\ell_1$-norm in the sparse inversion case and a Ginzburg--Landau energy in the binary classification problem.}
We introduce the particular problems and subgradient flows in Sections~\ref{Sec_Sparse_Inv} and \ref{Sec_Class}, respectively.
In both these cases, a direct (explicit) numerical simulation of the flows is highly inefficient due to non-smoothnesses in both differential inclusions. However, both differential inclusions can be separated into a linear ODE \review{(the gradient of data misfit and linear-quadratic part of the regulariser)} and a non-smooth part \review{(the subgradient of the non-smooth part of the regulariser)}. The linear part alone can be solved analytically or (more reasonably) discretised at a high accuracy in a stable way. \review{The non-smooth part alone can, in both cases, be solved analytically and results in a thresholding -- while this is natural for the $\ell_1$-norm, the Ginzburg--Landau energy is modelled in a way that achieves this property}.
This thresholding naturally becomes a randomised thresholding within the stochastic approximation. \review{Indeed, 
\begin{itemize}
    \item in sparse inversion, we subtract or add a random variable to move towards zero by a random amount, and
    \item in binary classification, we subtract or add a random variable to move towards $-1$ or $1$ (the one that is already closer) by a random amount.
\end{itemize}
We discuss this thresholding more particularly in Subsections~\ref{subsec_stochAppro_l1} and \ref{subsec_stochAppro_AC}, respectively.}

For both these problems, (discrete-time) splitting techniques that separate linear and non-smooth part to follow the linear flow and threshold are known and popular in practice: forward-backward splitting for sparse inversion and Merriman--Bence--Osher (MBO) for classification. In the present work, we analyse the non-discretised stochastic approximation leading to randomised versions of these splitting techniques.

\paragraph{Stochastic?} The introduction of stochasticity appears a bit artificial in this setting. Indeed, the algorithms mentioned above are fully deterministic.  We employ it for four different reasons: Firstly, the stochasticity allows us to define a Markovian homogeneous-in-time dynamical system. This not only simplifies the analysis, but also matches the discrete splitting algorithms\review{, which is also a Markovian homogeneous-in-time dynamical system}. Secondly, in the continuous setting, we cannot usually hope for convergence of such two-step methods to a single point. Analysing ergodicity and the convergence to a probability measure feels easier for a random process. This probability measure describes what is called \emph{implicit regularisation} in the machine learning context and sometimes also of interest from an uncertainty quantification perspective. Speaking about machine learning --  thirdly, stochastic optimisation methods have recently gained popularity in machine learning where non-convex problems need to be regularised and optimised. When sparsity shall be enforced, stochastic optimisation methods could profit from a randomised thresholding instead of a deterministic thresholding\review{, consider, e.g., \cite{pmlr-v162-mishchenko22b}}. Fourthly,  deep neural networks are often compositions of smooth functions and tresholdings. Ideas presented here could possibly be used to analyse randomisation in neural networks\review{; a connection that is particularly interesting as neural networks are popular tools in both, image reconstruction and (binary) classification}.

\subsection{Contributions and outline.}
The main contributions of the present work are:
\begin{itemize}
    \item We propose and study stochastic approximations of subgradient flows in sparse inversion and binary classification. The stochastic approximations are continuous\review{-time}, randomised models for usual splitting schemes.
    \item We analyse the long-time behaviour of the processes. Indeed, we see that the stochastic approximation is usually ergodic in case of the sparse inversion, while we need strict assumptions to show ergodicity in the classification case.
    \item We study the convergence of the stochastic approximations to the underlying subgradient flows, as the switching rate $\lambda \rightarrow \infty.$ Thus, we show that the stochastic approximations can approximate the subgradient flow at any accuracy.
    \item We illustrate our findings in a simple sparse estimation problem and \review{two} spatial classification problem\review{s}.
\end{itemize}
We consider non-smooth piecewise deterministic Markov processes, where the non-smoothness arises from the deterministic processes rather than only the stochastic switching/jumping (as opposed to, e.g., the Zigzag method, \cite{Bierkens}). Thus, the techniques employed in our analysis may be of independent interest.

This work is organised as follows. In Sections~\ref{Sec_Sparse_Inv} and \ref{Sec_Class}, we introduce and discuss the considered sparse inversion and classification problems, respectively, and derive our stochastic approximations. We show that all the considered processes \review{possess the so-called Feller property} and derive their generators in Section~\ref{Sec_Fell}. We use this information to show ergodicity in Section~\ref{sec_longtime} and the convergence of the stochastic approximation to their related subgradient flow in Section~\ref{sec_perturbed}. In Section~\ref{Sec_Illust} we illustrate our results in \review{numerical} examples, before concluding in Section~\ref{Sec_Conclusions}.
\section{Sparse inversion} \label{Sec_Sparse_Inv}
We commence with the sparse inversion problem. In the following, we introduce this problem, the subgradient flow, and its stochastic approximation.

We consider an optimisation problem of the form
\begin{equation} \label{Eq:Optprob}
    \mathrm{argmin}_{\theta \in X} \hPhi(\theta), 
\end{equation}
where
\begin{equation*}
    \hPhi(\theta) := \frac{1}{2}\|A \theta - b \|^2_2 + \|\theta\|_1 \quad (\theta \in X)
\end{equation*}
 and $X = \mathbb{R}^n$, $Y = \mathbb{R}^m \ni b$,  $A \in  \mathbb{R}^{m \times n}$, with $n, m \in \mathbb{N}:= \{1,2,3,\ldots\}$. In the following, we assume that $\mathrm{rank}(A) = n$.  We denote the two parts of the target function $\hPhi$ by $\dPhi := \frac{1}{2}\|A \cdot - y \|^2_2$ and $\rPhi := \| \cdot \|_1$. We abbreviate $\AtA := A^TA$ and $\Atb := A^Tb$.
\review{As the optimisation problem \eqref{Eq:Optprob} is non-smooth, we need to introduce a {subdifferential} to discuss possible solutions. Throughout this work, we consider the \emph{proximal subdifferential} which is defined for a potential $\Phi: X \rightarrow \mathbb{R}$ by $$
\partial \Phi(x) := \left\lbrace v \in X : \exists \varepsilon, r >0  : \langle v, y-x \rangle \leq \Phi(y) - \Phi(x) + r \|x-y\|^2_2 \  (y \in Y: \|x-y\|_2 < \varepsilon)\right\rbrace \ (x \in X).
$$}
To solve problem \eqref{Eq:Optprob}, we need to find a stationary point of $\hPhi$, e.g. a point $\theta_* \in X$ at which
\begin{equation}
\nabla \dPhi(\theta_*) + g    =\AtA \theta_* - \Atb + g = 0,
\end{equation}
where $g \in \partial \rPhi(\theta_*)$.
The sub\review{differential} is given by $\partial \rPhi(\theta) := G(\theta_1) \times \cdots \times G(\theta_n)$, where $$ G(\theta'):= \begin{cases} \{1\}, &\text{ if } \theta' > 0, \\ \{-1\}, &\text{ if } \theta' < 0, \\ [-1,1], &\text{ if } \theta' = 0\end{cases} \qquad \qquad (\theta' \in \mathbb{R}).$$

A natural way to find this stationary point $\theta_*$ \review{is} to follow the  subgradient flow, which we call \emph{sparse inversion flow}:

\begin{equation} \label{EQ:full_subgrad}
   \frac{\dt{\zeta}(t)}{\dt t} \in - \frac{1}{2}\nabla \dPhi(\zeta(t)) - \frac{1}{2}\partial \rPhi(\zeta(t)) \quad (t > 0), \qquad \zeta(0) = \zeta^{(0)}, 
\end{equation}
for some initial value $\zeta^{(0)} \in X$. In this and other dynamical systems, we sometimes emphasise the dependence on the initial value by writing $(\zeta(t, \zeta^{(0)}))_{t \geq 0} := (\zeta(t))_{t \geq 0}$.  We scale the right-hand side in \eqref{EQ:full_subgrad} by $1/2$ for convenience.
This dynamical system has a unique, locally absolutely continuous solution, see, e.g. Theorems 2.9 and 3.2 in \cite{Marcellin2006}. The solution is also locally Lipschitz continuous and $(\zeta(t))_{t \geq 0}$ converges to the stationary point of $\hPhi$; we discuss this in Proposition~\ref{prop_det_conv}.

Essentially based on this subgradient flow, several techniques have been proposed in the past decades to find $\theta_*$, such as subgradient descent \cite{Shor1985} or the already mentioned forward-backward splitting  \cite{Combettes2011,Goldstein}.
Indeed, the basic subgradient descent method follows the subgradient flow by discretising the system with a forward Euler discretisation
$$
\zeta^{(k+1)} \leftarrow \zeta^{(k)} - \frac{h_k}{2}\left(\nabla\dPhi(\zeta^{(k)}) + g_k \right), \qquad  g_k \in \partial \rPhi(\zeta^{(k)}), \qquad (k=0, 1,\ldots),
$$
where $(h_k)_{k=0}^\infty$ is a sequence of step sizes.
This method has disadvantages in practice that are easily explained by the missing Lipschitz continuity of $\partial \rPhi$. The  forward-backward splitting algorithm is considerably better: The idea is that splitting the system into two parts \review{such that each of them can be treated easily}. Thus, forward-backward splitting proceeds by alternately minimising $\dPhi$ and $\rPhi$ using gradient and proximal step, respectively, see also Algorithm~\ref{algo_fwbwS}. In the algorithm, we denote the signum function $\mathrm{sgn}(x) := \mathbf{1}[x > 0] -  \mathbf{1}[x < 0]$, where $\mathbf{1}[\mathrm{true}] = 1$, $\mathbf{1}[\mathrm{false}] = 0$,  the positive part $f_+ = \max\{f, 0\}$ of a function $f$, and the Hadamard (element-wise) product $\odot$. Here and throughout the rest of this work, we sometimes apply scalar functions to vectors in which case we mean component-wise application. The `$\mathrm{argmin}$' in the last line of the algorithm is the proximal operator of $\rPhi$ with stepsize $h_k$.

\begin{algorithm} 
\caption{Forward-backward splitting for \eqref{Eq:Optprob}} \label{algo_fwbwS}
\begin{algorithmic}
\Require initial value $\zeta^{(0)}$, sequence of positive step sizes $(h_k)_{k=0}^\infty$
\For{$k=1,\ldots$}
\State $\hat\zeta \leftarrow \zeta^{(k-1)} - h_k\nabla \dPhi(\zeta^{(k)}) $
\State $\zeta^{(k)} \leftarrow \argmin_{\zeta'} \{h_k\rPhi(\zeta') + \frac{1}{2} \|\zeta' - \hat\zeta\|^2\}$ = $\mathrm{sgn}(\hat\zeta)\odot(|\hat\zeta| - h_k)_+$ \Comment{proximal operator of $\rPhi$}
\EndFor 
\end{algorithmic}
\end{algorithm}
In the following, we utilise this splitting idea to derive our stochastic approximation.
\subsection{The stochastic approximation} \label{subsec_stochAppro_l1}
We now split the dynamical system \eqref{EQ:full_subgrad} into two subflows:
\begin{align}
     \frac{\dt{\zeta^\dt}(t)}{\dt t}&= - \nabla\dPhi({\zeta}^\dt(t)) \quad (t > 0), \qquad \zeta^\dt(0) = \zeta^{\dt,0},  \label{eq:ODE_dt} \\ 
      \frac{\dt{\zeta^\rg}(t)}{\dt t} &\in - \partial \rPhi(\zeta^\rg(t)) \quad (t > 0), \qquad \zeta^\rg(0) = \zeta^{\rg,0}\label{eq:ODE_rg}, 
\end{align}
for initial values $\zeta^{\dt,0}, \zeta^{\rg,0} \in X$. The linear differential equation \eqref{eq:ODE_dt} has a closed form solution. We have: 
\begin{equation*}
   \zeta^\dt(t) = \AtA^{-1}\Atb + \exp(-t\AtA)(\zeta^{\dt,0} - \AtA^{-1}\Atb) \qquad (t \geq 0),
   \end{equation*}
  where, of course, $\exp(M) := \sum_{j=0}^\infty \frac{1}{j!}M^j$ is the matrix exponential of $M \in \mathbb{R}^{n \times n}$.
   For the subgradient flow \eqref{eq:ODE_rg}, we can find the solution: 
   \begin{equation*}
         \zeta^\rg(t) = \mathrm{sgn}(\zeta^\rg)\odot(|\zeta^\rg| - t)_+.
   \end{equation*}
For the uniqueness of this solution, we refer the reader again to \cite{Marcellin2006}. Note that $(\zeta^\rg(t))_{t \geq 0}$ \review{equals} the proximal operator \review{of the $\ell^1$-norm} in Algorithm~\ref{algo_fwbwS}.

In the following, we study the stochastic approximation of \eqref{EQ:full_subgrad} that alternately follows the flows  \eqref{eq:ODE_dt} and \eqref{eq:ODE_rg}. To control which subflow is employed at which time, we employ the continuous-time Markov process $(\bsi(t))_{t \geq 0}$ that we have defined in Subsection~\ref{subsec_Problem_Setting}. 
The \emph{stochastic approximation} $(\bsi(t), \theta(t))_{t \geq 0}$ of $(\zeta(t))_{t \geq 0}$ is then defined by
\begin{equation} \label{EQ:stochastic_approximation_sparse_inv}
    \frac{\dt \theta(t)}{\dt t} \in  \begin{cases} \{- \nabla\dPhi(\theta(t)) \}, &\text{if } \bsi(t) = 0, \\
    - \partial \rPhi(\theta(t)), &\text{if } \bsi(t) = 1
    \end{cases} \qquad (t > 0), \qquad \theta(0) = \theta^{(0)},
\end{equation}
where $\theta^{(0)} \in X$ is an initial value. This stochastic process is well-defined with probability 1, as it is a piecewise construction of well-defined functions on non-empty intervals -- note that the waiting times in between two jumps of $(\bsi(t))_{t \geq 0}$ is exponentially distributed and, thus, strictly positive with probability 1. 

\begin{remark} \label{rem_thresholding} 
Let $I := (\widehat{T}_{k-1}, \widehat{T}_k]$ be an interval such that $\bsi(\cdot)|_I \equiv 1$, i.e. $(\theta(t))_{t \in I}$ follows the sub\review{differential} of $\rPhi$. We set $\delta T := \widehat{T}_{k}-\widehat{T}_{k-1}$. Within the interval $I$, the subgradient flow performs the following action:
$$
 \theta_i(\widehat{T}_k) = \begin{cases}(\theta_i(\widehat{T}_{k-1})-\delta T)_+, &\text{ if } \theta_i \geq 0, \\ -(-\theta_i(\widehat{T}_{k-1})-\delta T)_+, &\text{ otherwise.}\end{cases}
$$
This is a randomised thresholding operation, where the threshold $\delta T \sim \mathrm{Exp}(\lambda)$. Thus, the dynamical system uses a randomised thresholding to approach the minimiser of the sparsity enforcing regulariser $\rPhi$.
\end{remark}

We give a simple example for a subgradient flow $(\zeta(t))_{t \geq 0}$ and its stochastic approximation $(\theta(t))_{t \geq 0}$ in Figure~\ref{fig:simple_example}. For small $\lambda$, we are able to recognise the structure of the stochastic approximation: smooth curves targetting the minimiser of $\dPhi$ and linear curves pointing towards the minimiser of $\rPhi$. For large $\lambda$, we see that the stochastic approximation approaches the underlying subgradient flow.

\begin{figure}[tbh]
    \centering
    \includegraphics[scale = 0.65]{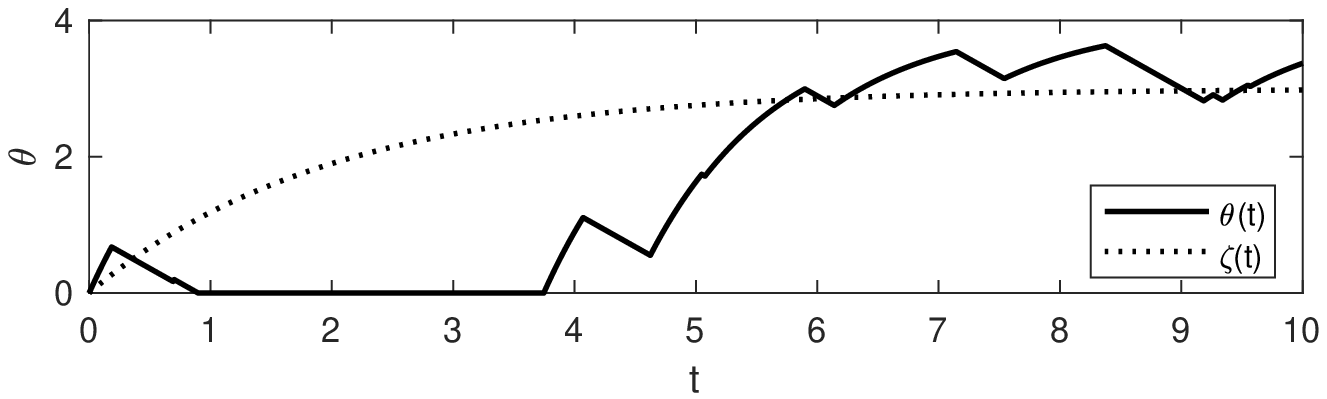}
    \includegraphics[scale = 0.65]{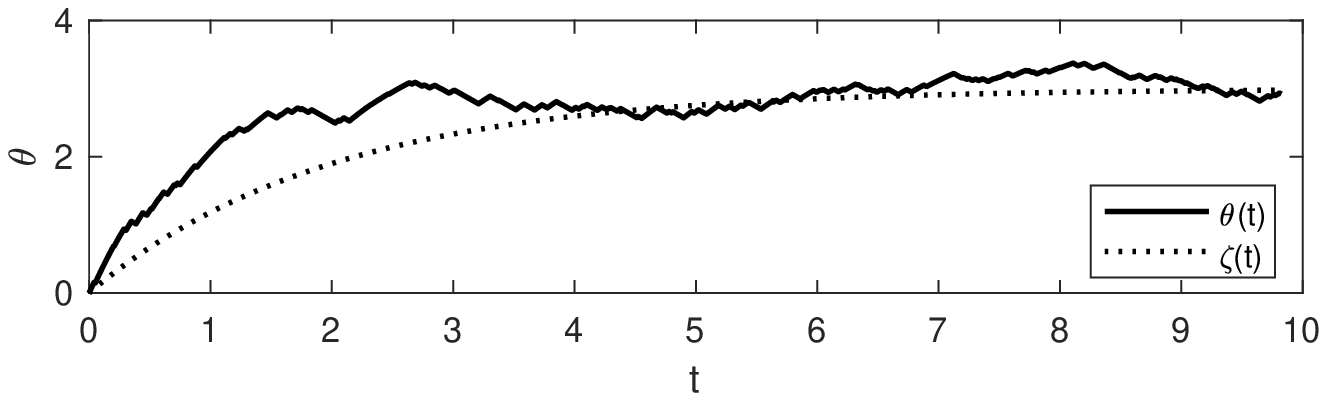}
    \includegraphics[scale = 0.65]{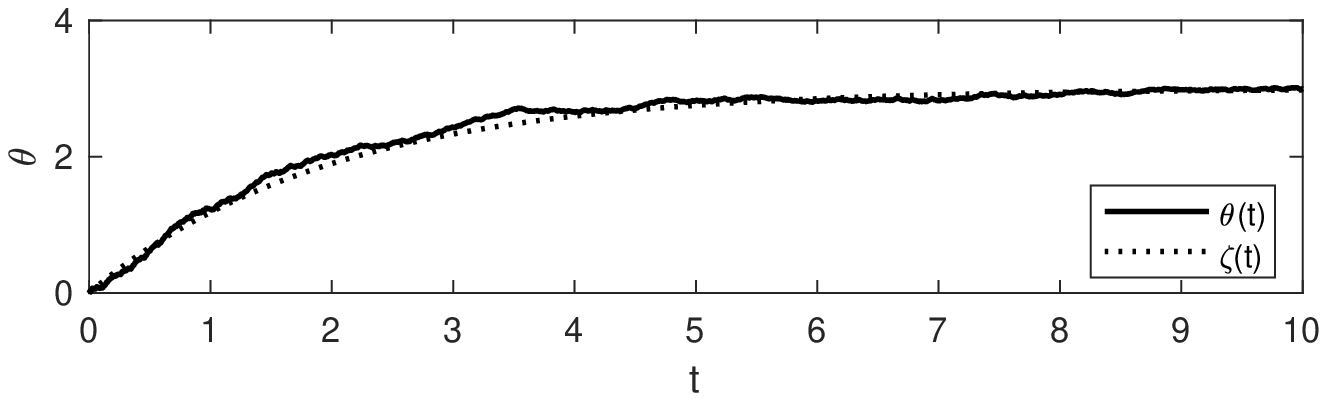}

    \caption{Subgradient flow \eqref{EQ:full_subgrad} and its stochastic approximation \eqref{EQ:stochastic_approximation_sparse_inv} minimising $\frac{1}{2}(\theta-4)^2 + |\theta|$, with stationary point $\theta_*$. In the different figures, we have $\lambda = 2.5, 25, 250$ (from top to bottom). The shown paths have all been determined analytically. For a deeper discussion of this problem, see Subsection~\ref{subsec:simple_exmpl}.}
    \label{fig:simple_example}
\end{figure}

\review{We finish this section by mentioning several related stochastic splitting methods. The idea of randomly choosing waiting times in between proximal steps appears in the  discrete-time method proposed by Mishchenko et al. \cite{pmlr-v162-mishchenko22b}. There, the proximal step is skipped with a high  probability leading to a random amount of time that is spent in the gradient flow part. The tresholding is still deterministic.  More traditional stochastic splitting methods have been discussed in several works, e.g., \cite{Atchade,Rosasco}. In these works, similar to Stochastic Gradient Descent \cite{RobbinsMonro}, only random approximations of  $\nabla\dPhi$ are available but not the gradient itself. The stochasticity here arises from those noisy gradients.}
 \section{Classification} \label{Sec_Class}
 Next, we move on to the classification problem. We introduce the problem, the subgradient flow, and its stochastic approximation. Moreover, we show well-posedness of the subgradient flow. 
 
 We consider an optimisation problem of the form 
 \begin{equation*} 
   \mathrm{argmin}_{\eta \in X} \tilde{\Phi}(\eta),
 \end{equation*}
 where 
 \begin{equation*}
     \tilde{\Phi}(\eta) := \frac{\alpha}{2}\|P_{X'}\eta- d \|_2^2 + \frac{\varepsilon}{2}\|\nabla' \eta\|_2^2 + \varepsilon^{-1}\sum_{i=1}^n W(\eta_i) \quad (\eta \in X)
 \end{equation*}
 and $\alpha, \varepsilon >0$, $X'$ is a subspace of $X$, $d \in X'$, $P_{X'}: X \rightarrow X'$ is the canonical projection on\review{to} $X'$, $\nabla'$ is some discrete gradient, and $W: \mathbb{R} \rightarrow \mathbb{R}$ is the non-smooth \emph{double-well potential}
 $$
 W(\eta) = \begin{cases} 
 - \eta - 1, &\text{ if } \eta \leq -1, \\
-\frac{\eta^2}{2} +\frac{1}{2}, &\text{ if } \eta \in (-1,1), \\
\eta-1, &\text{ if } \eta \geq 1,
\end{cases}
 $$
see also the top of Figure~\ref{fig:doubWell}, where we plot the double-well potential. 

Here,  $d \in X'$ represents a possibly noisy training data set which we intend to use to classify the remainder of the data in $\eta$. The double-well potential is responsible for leading the entries of $\eta$ to either $1$ or $-1$; indeed constructing a classification of the entries of $\eta$. The penalisation of the gradient implies an additional localisation by producing neighbouring components in $\eta$ that are (about) constant $1$ or $-1$. The meaning of `neighbouring' is determined by the choice of the gradient, see also Remark~\ref{rem:anylapla}. A similar double-well potential has been employed by \cite{BuddVG,BuddvGL} -- they actually used a related double-obstacle formulation. \review{This double-well potential is rather non-classical, we employ it as its subgradient flow is a thresholding operation that we can evaluate analytically; see Subsection~\ref{subsec_stochAppro_AC} below. } 
\review{The parameter} $\varepsilon > 0$ controls whether the process is closer to a diffusion $(\varepsilon \gg 0)$ or to a strict binary classification $(\varepsilon \approx 0)$. The part $\frac{\varepsilon}{2}\|\nabla' \eta\|_2^2 + \varepsilon^{-1}\sum_{i=1}^n W(\eta_i)$ in \eqref{eq:GL+Fid} is a \review{so-called} \emph{Ginzburg--Landau energy}\review{, see, e.g. \cite{Ting_2009}}.
 
 The gradient flow of a Ginzburg--Landau energy is an \emph{Allen--Cahn equation}.
In our discrete setting, we obtain the following \emph{discrete Allen--Cahn equation (with fidelity forcing)}:
 \begin{equation} \label{eq:GL+Fid}
     \frac{\dt \xi(t)}{\dt t} \in - \frac12 \PtP \xi(t) + \frac12\Ptd + \frac{\varepsilon}{2}\triangle' \xi(t) - \frac{1}{2\varepsilon}\partial W(\xi(t)) \qquad (t \geq 0), \qquad \qquad  \xi(0) = \xi^{(0)},
 \end{equation}
 where we abbreviate $\PtP := \alpha P_{X'}^TP_{X'}$, $\Ptd := \alpha P_{X'}^Td$, $\triangle' := - \nabla'^T\nabla'$ is a certain discrete Laplacian, and the scaling by $1/2$ is again a matter of convenience.   The sub\review{differential} of $W$ is given by
 $$
 \partial W(\eta') = \begin{cases} 
 \{-1 \}, &\text{ if } \eta' < -1, \\
 \{-\eta \}, &\text{ if } \eta' \in (-1,1), \\
 \{ 1 \}, &\text{ if } \eta' > 1, \\
 [-1 , 1],  &\text{ if } \eta' \in \{ -1, 1\}
 \end{cases}
 \qquad \qquad (\eta' \in X).
 $$
 Note that $W$ is not a convex function; \review{thus, here it is especially important that we compute the proximal subdifferential}, see \cite[Section 1]{Marcellin2006}. We show in Proposition~\ref{prop_AC_exist} below that this sub\review{differential} coincides with other usual sub\review{differentials} by proving that $W$ is \emph{primal lower nice}.
 We illustrate the sub\review{differential} in the bottom of Figure~\ref{fig:doubWell}.
 \begin{remark} \label{rem:anylapla}
Up to here and throughout the rest of this work, we do not specify properties of the discrete gradient $\nabla'$ and the discrete Laplacian $\triangle'$. In the following we only require  $\triangle'$ to be a symmetric, strictly negative definite matrix. It especially does need to be neither a localised operation, nor a discretised differential operator: fractional operators, integral operators, diffusions on graphs are allowed, as are different discretisation strategies.
 \end{remark}
 \begin{figure}
     \centering
     \includegraphics[scale = 0.7]{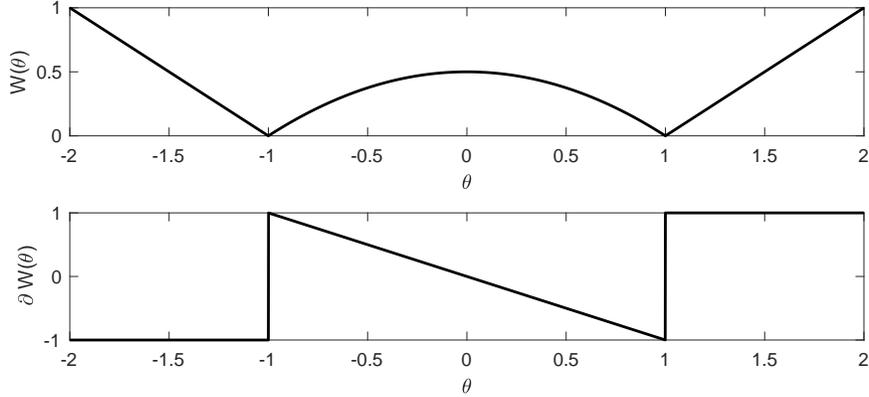}
     \caption{Double-well potential $W$ (top) and its sub\review{differential} $\partial W$ (bottom).}
     \label{fig:doubWell}
 \end{figure}
 
 \review{Before moving on to the analysis of the discrete Allen--Cahn equation, we briefly give more context: The Allen--Cahn equation has first been proposed to model the motion of an antiphase boundary \cite{AC_orig}. The connection to the Ginzburg--Landau energy has been discussed by \cite{Saner1997}. The numerical treatment of the Allen--Cahn equation has been discussed by, e.g., \cite{Feng,LI20101591}. Finally, the idea of using Allen--Cahn in classification and image segmentation has become popular in recent years; see, e.g., \cite{BENES2004,BuddVG,BuddvGL,Lee2022}.}
 
 \subsection{Well-posedness of the discrete Allen--Cahn equation.}
We now prove existence, uniqueness, and continuous dependency on the initial value for the solution to \eqref{eq:GL+Fid}. To the best of our knowledge, this has not been done for this particular dynamical system. \review{We proceed by proving that the double-well potential we employ is convex up to a quadratic functional.}

\begin{proposition} \label{prop_AC_exist}
The differential inclusion \eqref{eq:GL+Fid} has a unique solution that is locally Lipschitz continuous. Hence, the problem is well-posed.
\end{proposition}
\begin{proof}
One way to show the existence and the uniqueness of the solution of \eqref{eq:GL+Fid} is by proving that $W: \mathbb{R} \rightarrow \mathbb{R}$ is primal lower nice. In this case, Theorems 2.9 and 3.2 in \cite{Marcellin2006} apply and prove the assertion of the proposition.

A function $f: \mathbb{R} \rightarrow \mathbb{R}$ is primal lower nice at \review{$u_0$}, if there are $s_0, c_0, Q_0 > 0$ such that for all $x \in B(u_0, s_0) := \{u : |u_0 - u| \leq s_0 \}$ and for all $q \geq Q_0$, $v \in \{ u \in  \partial f(x) : |u | \leq c_0q\}$, one has 
$$
f(y) \geq f(x) +  v (y-x) - \frac{q}{2}(y-x)^2.
$$
We first show that we can represent the function $W $ as the sum of a convex function and a negative square term. Then, we show that such combinations are always primal lower nice.
Indeed, we can write $2W(\eta) = g(\eta) - \eta^2,$
where  $$ g(\eta) = \begin{cases} 
 - 2\eta - 2 + \eta^2, &\text{ if } \eta \leq -1, \\
1, &\text{ if } \eta \in (-1,1), \\
2\eta-2 + \eta^2, &\text{ if } \eta \geq 1.\end{cases}
$$
One can easily see that $g$ is continuous and convex. We now show that $2W(\eta)$ is primal lower nice:
\begin{align*}
    g(y) - y^2 &\geq g(x) - x^2 + (v - 2x)(y-x) - \frac{q}{2}(y-x)^2 \\&= g(x) - x^2 + v (y-x) - 2xy + 2x^2 - \frac{q}{2}(y-x)^2 \\
    &= g(x) + v (y-x) + x^2  - 2xy  - \frac{q}{2}(y-x)^2
\end{align*}
which is equivalent to
$$
g(y) \geq g(x) + v(y-x)  + \left(1 - \frac{q}{2}\right)(y-x)^2,
$$
which holds for sufficient $s_0, c_0, Q_0$ as $g$ is convex.
\end{proof}

There are again multiple ways to discretise the differential inclusion \eqref{eq:GL+Fid} via explicit, semi-implicit, and fully implicit methods, see, e.g. \cite{BuddvGL,LI20101591}. Again, we make the observation that we can split the right-hand side of \eqref{eq:GL+Fid} into two parts that are relatively simple to solve separately. We discuss this and the implied stochastic approximation in the next subsection.

\subsection{The stochastic approximation} \label{subsec_stochAppro_AC}
The differential inclusion \eqref{eq:GL+Fid} contains a linear/smooth and a non-linear/non-smooth part. We denote the potentials of these two parts by $\Phi_{\lin}(\eta) := \frac{1}{2}\|P_{X'}\eta-d \|_2^2 + \frac{\varepsilon}{2}\|\nabla' \eta\|_2^2$ and $\Phi_{\dw}(\eta) := \varepsilon^{-1}\sum_{i=1}^n W(\eta_i)$. The corresponding dynamical systems are given by
\begin{align}
    \frac{\dt \xi^{\lin}(t)}{\dt t} &= \nabla\Phi_{\lin}(\xi^{\lin}(t)) =  - \PtP \xi^{\lin}(t) + \Ptd + \varepsilon\triangle' \xi^{\lin}(t) &(t\geq 0), \quad \qquad  \label{eq_AC_diff} \xi^{\lin}(0) = \xi^{\lin,0},\\
    \frac{\dt \xi^{\dw}(t)}{\dt t} &\in \partial\Phi_{\dw}(\xi^{\dw}(t)) = -\varepsilon^{-1}\partial W(\xi^{\dw}(t)) &(t\geq 0), \quad \qquad  \xi^{\dw}(0) = \xi^{\dw,0}. \label{eq_AC_thresh}
\end{align}
The linear part can, again, be solved analytically:
$$
\xi^{\lin}(t) = (\PtP - \varepsilon\triangle')^{-1} \Ptd + \exp(-t(\PtP-\varepsilon\triangle'))(\xi^{\lin,0} - (\PtP - \varepsilon\triangle')^{-1} \Ptd) \qquad (t \geq 0).
$$
The non-smooth part is a bit more involved, but we can find an analytical solution also in this case. We write the solution coordinate\review{-}wise. For $i \in  \{1,\ldots,n\}$, we have
$$
\xi^{\dw}_i(t) = \begin{cases} \min\{\xi^{\dw,0}_i + \varepsilon^{-1}t,  -1\}, &\text{ if } \xi^{\dw,0}_i < -1, \\
\max\{\varepsilon^{-1}\exp(t)\xi^{\dw,0}_i, -1 \}, &\text{ if } \xi^{\dw,0}_i \in [-1, 0), \\
0, &\text{ if } \xi^{\dw,0}_i = 0, \\ 
\min\{\varepsilon^{-1}\exp(t)\xi^{\dw,0}_i, 1 \}, &\text{ if } \xi^{\dw,0}_i \in (0,1], \\
\max\{\xi^{\dw,0}_i - \varepsilon^{-1}t,  1\}, &\text{ if } \xi^{\dw,0}_i > 1.
\end{cases}
$$
which is a unique solution of the differential inclusion. The uniqueness can be proven in the same way as we \review{show}  Proposition~\ref{prop_AC_exist}.

We obtain the \emph{stochastic approximation} $(\bsi(t), \eta(t))_{t \geq 0}$ of the discrete Allen--Cahn equation $(\zeta(t))_{t \geq 0}$ by alternating $(\xi^{\lin}(t))_{t \geq 0}$ and $(\xi^{\dw}(t))_{t \geq 0}$. We obtain:
\begin{equation} \label{EQ:stochastic_approximation_AC}
    \frac{\dt \eta(t)}{\dt t} \in  \begin{cases} \{- \nabla\lPhi(\eta(t)) \}, &\text{if } \bsi(t) = 0, \\
    - \partial \wPhi(\eta(t)), &\text{if } \bsi(t) = 1
    \end{cases} \qquad (t > 0), \qquad \eta(0) = \eta^{(0)},
\end{equation}
where $\eta^{(0)} \in X$ is an initial value. The stochastic approximation is, again, well-defined.
\begin{remark}
What we discuss in Remark~\ref{rem_thresholding} also holds true in this setting: we randomise the waiting times between two switches in $(\bsi(t))_{t \geq 0}$. Thus, the  subgradient flow with respect to the double-well potential $\wPhi$, again, leads to a randomised thresholding.
\end{remark}
A multitude of splitting schemes for  classification and segmentation tasks of this type have been proposed in the literature.
The idea of alternating diffusion and thresholding  steps, which we use here as well, is particularly iconic:  It is the underlying principle of the MBO scheme \cite{MBO}. There, this   idea has been used in the context of mean curvature flows. Its use in segmentation and classification has been considered by, e.g., \cite{BuddVG,BuddvGL,Cucuringu2021,Esedoglu06}. The connection that we make here between MBO and Allen--Cahn has been thoroughly discussed by \cite{BuddVG} in the setting of graph diffusions. Other splittings are also possible, see, for instance  \cite{Bertozzi-Flenner} and the discussion therein. \review{Finally, we note that through our stochastic approximation, we introduce a certain kind of noise into the Allen--Cahn equation. This introduced piecewise deterministic noise produces in some sense a stochastic Allen--Cahn equation. This stochastic Allen--Cahn equation, however, is rather different from the usual stochastic Allen--Cahn equation, see e.g., \cite{BERTACCO2021112122}, which is driven by white noise. }

\section{Feller processes and their generators} \label{Sec_Fell}
In the following, we discuss certain basic properties of the (sub)gradient flows $(\zeta(t))_{t \geq 0}$, $(\zeta^\dt(t))_{t \geq 0}$, $(\zeta^\rg(t))_{t \geq 0}$, $(\xi(t))_{t \geq 0}$, $(\xi^\lin(t))_{t \geq 0}$, $(\xi^\dw(t))_{t \geq 0}$, the process $(\bsi(t))_{t \geq t}$, and the stochastic approximations $(\theta(t))_{t \geq 0}$ and $(\eta(t))_{t \geq 0}$. We are particularly interested in showing that the processes are Feller and in finding their infinitesimal generators. These results and generators are vital to discuss longtime behaviour and convergence of the stochastic approximations.

We start with the definition of these terms. Then, we discuss the Feller property and the infinitesimal \review{generators} of the \review{(proper)} ODEs and $(\bsi(t))_{t \geq 0}$. Moreover, we explain how we obtain the infinitesimal generators for the subgradient processes. Ultimately, we are able to find the generators of the stochastic approximations $(\theta(t))_{t \geq 0}$ and $(\eta(t))_{t \geq 0}$ in Subsections~\ref{subsec_Fell_sparse} and \ref{subsec_Fell_class}\review{, respectively}.

\begin{definition} A homogeneous-in-time stochastic process $(x(t))_{t \geq 0}$ is \emph{Feller}, if its Markov transition kernel $P_t(\cdot | x') := \mathbb{P}(x(t) \in \cdot | x(0) = x')$ $(x' \in X, t \geq 0)$ maps from $C_0(X)$ into $C_0(X)$ \review{and} satisfies the following three assumptions:
\begin{itemize}
    \item[(i)] $\|P_t f\|_{\infty} \leq \|f\|_{\infty}$ $(f \in C_0(X), t \geq 0)$,
    \item[(ii)] $P_{t+s} = P_t \circ P_s$ $(s,t \geq 0)$, and
    \item[(iii)] $\lim_{t \rightarrow 0} \|P_t f - f \|_{\infty} = 0$ $(f \in C_0(X))$.
\end{itemize}
Moreover, we define the \emph{(infinitesimal) generator} of $(x(t))_{t \geq 0}$ by
$$
\mathcal{A}f := \lim_{t \rightarrow 0} \frac{P_t f - f}{t},
$$
which is well-defined for any Feller process and $f$ in an appropriate subspace of $C_0(X)$.
\end{definition}
In this definition, we denoted $C_0(X) := C_0^0(X)$ which is one of the spaces 
$$C_0^k(X) := \left \lbrace g : X \rightarrow \mathbb{R} : g \text{ is continuous and $k$-times continuously differentiable } \lim_{\|\theta\|\rightarrow \infty}|f(\theta)| = 0\right \rbrace,$$ for $k \in  \mathbb{N} \cup \{0\}$. Moreover, for some $f \in C_0(X)$, we define $P_tf := \int_X f(x) P_t(\mathrm{d}x|\cdot)$ \review{to be the expected value of $f(z)$, where $z \sim P_t$}.

For some of the processes mentioned above, it is easy to show the Feller property and derive the infinitesimal generator. Those are
\begin{enumerate}
    \item the ODEs $(\zeta^\dt(t))_{t \geq 0}$ and $(\xi^\lin(t))_{t \geq 0}$.  The generator of a well-defined ODE $\frac{\dt x}{\dt t} = \varphi(x)$ is given by
$$
\mathcal{A}_{\varphi}f = \langle \nabla f, \varphi \rangle
$$
for $f \in D := C^1_0(X) \subseteq C_0(X)$.
\item the continuous-time Markov process $(\bsi(t))_{t \geq 0}$. The Markov transition kernel of $(\bsi(t))_{t \geq 0}$ is given by
$$
\mathbb{P}(\bsi(t) = i \mid \bsi(0) = i_0) = \frac{1- \exp(-2\lambda t)}{2} + \exp(- 2\lambda  t)\mathbf{1}[i = i_0] =: I_t(\{ i \} \mid i_0) \qquad  (i, i_0 \in \{0, 1\}),
$$
see \cite[Lemma 5]{Latz2021} for details. The infinitesimal generator is the transition rate matrix that we give in Subsection~\ref{subsec_Problem_Setting}, it has domain $D = \mathbb{R}^{\{0,1\}} =  \mathbb{R}^2$.
\end{enumerate}
The more challenging cases are the (proper) subgradient flows $(\zeta(t))_{t \geq 0}$, $(\zeta^\rg(t))_{t \geq 0}$, $(\xi(t))_{t \geq 0}$, and $(\xi^\dw(t))_{t \geq 0}$. We know from the discussion in Sections~\ref{Sec_Sparse_Inv} and \ref{Sec_Class} that all of these paths are uniquely defined and locally Lipschitz continuous. This already implies that the processes are Feller. However, the paths are  not necessarily differentiable. Thus, also $t \mapsto P_tf$ is not differentiable for an \review{arbitrary} $f \in C^1_0(X)$ \review{that is not in the correct subspace of $C^1_0(X)$. Instead of determining this subspace and then taking the derivatives, we discuss below why it is sufficient to work with right semi-derivatives when discussing the generator}. 

\review{A function $g: [0, \infty) \rightarrow X$ is \emph{right semi-differentiable}, if the limit } 
$$
\lim_{h \downarrow 0}\frac{g(t+h)-g(h)}{h}  =: \frac{\dt g(t)}{\dt t^+}
$$
\review{exists for $t \geq 0$. The limit $\frac{\dt g(\cdot)}{\dt t^+}$ is the \emph{right semi-derivative} of $g$.}
Then, we indeed have that the infinitesimal operator $\mathcal{A}$ of a Markov transition operator $P_t$ is given by
$$\mathcal{A}f = \frac{\dt  P_tf}{\dt t^+},$$ 
\review{as in the existence proof of the infinitesimal generator, the right semi-derivatives coincide with the derivatives for test functions $f$ from the appropriate subspace of $C_0^1(X)$, see Theorem 17.6 in \cite{Kallenberg}.}
Let now $P_t$ be the Markov transition operator for some process $(x(t))_{t \geq 0}$ that is a \review{right semi-differentiable}, Lipschitz continuous, deterministic Feller process. Then, $P_t( \cdot |x') = \delta(\cdot - x(t; x'))$ for $x' \in X$. For $f \in C_0^1(X)$, we have
$$
\mathcal{A}f = \lim_{t \downarrow 0} \frac{P_t f - f}{t} = \frac{\dt}{\dt t^+}\Big\lvert_{t=0} f(x(t)) = \left\langle \nabla f, \frac{\dt x(t)}{\dt t^+}\Big\lvert_{t=0} \right\rangle.
$$
Thus, we obtain the infinitesimal generators for $(\zeta(t))_{t \geq 0}$, $(\zeta^\rg(t))_{t \geq 0}$, $(\xi(t))_{t \geq 0}$, and $(\xi^\dw(t))_{t \geq 0}$ \review{after showing right semi-differentiability and }by computing their right semi-derivatives with respect to $t$. We do this in the next two subsections, in which we also derive the infinitesimal generators for the stochastic approximations $(\theta(t))_{t \geq 0}$ and $(\eta(t))_{t \geq 0}$.

\subsection{Sparse inversion} \label{subsec_Fell_sparse}
We now discuss the right semi-derivatives of the processes $(\zeta(t))_{t \geq 0}$ and $(\zeta^\rg(t))_{t \geq 0}$.
\begin{proposition} \label{Prop:ODEs}
The functions $(\zeta(t))_{t \geq 0}$ and $(\zeta^\rg(t))_{t \geq 0}$ are right semi-differentiable and have the following semi-derivatives:
\begin{align*}
  \frac{\dt \zeta(t)}{\dt t^+} &:= -\frac12 \nabla  \Phi_\dt(\zeta(t)) + \frac12 G(\zeta(t)),  \\
  G_i(\zeta) &:=  \begin{cases} 1, &\text{ if } \zeta_i < 0, \\ 
   -1, &\text{ if } \zeta_i > 0, \\
   (\nabla\dPhi(\zeta))_i, &\text{ if } \zeta_i = 0, (-\nabla \dPhi(\zeta))_i \in [-1,1], \\
   -1, &\text{ if } \zeta_i = 0, (-\nabla \dPhi(\zeta))_i > 1, \\
   1, &\text{ if } \zeta_i = 0, (-\nabla \dPhi(\zeta))_i < -1
   \end{cases} \qquad \qquad (i \in \{1,\ldots,n\}),\\
   \frac{\dt \zeta^\rg(t)}{\dt t^+}  &:= - \mathrm{sgn}(\zeta^\rg(t)),
\end{align*}
for $t \geq 0$, respectively.
\end{proposition}
\begin{proof}
Without loss of generality, we do the following computations only for $(\zeta(t))_{t \geq 0}$. 
Let $i \in \{1,\ldots,n\}$ and $t \geq 0$.
If $\zeta_i(t) > 0$, we have 
$$
 2 \cdot \frac{\dt \zeta(t)_i}{\dt t} = -\AtA_i \zeta(t) + \Atb_i - 1
$$
and, similarly,
$$
 2 \cdot  \frac{\dt \zeta(t)_i}{\dt t} = -\AtA_i \zeta(t) + \Atb_i + 1
$$
in case $\zeta_i(t) < 0$, as $\zeta_i(t)$ is continuously differential in those cases and satisfies uniquely the differential inclusion. Let now $\zeta_i(t) = 0$. Moreover, let $-\AtA_i \zeta(t) + \Atb_i > 1$, then
$
\zeta_i(t+h)> 0
$
for all $h> 0$ sufficiently small and thus
 $ 2 \cdot \frac{\dt \zeta(t+h)_i}{\dt t} = -\AtA_i \zeta(t+h) + \Atb_i - 1$,
which is continuous in $h$. Thus,
 $$ 2 \cdot \frac{\dt \zeta(t)_i}{\dt t^+} = -\AtA_i \zeta(t) + \Atb_i - 1.$$
In the same way, if $-\AtA_i \zeta(t) + \Atb_i < -1$, we have
 $$ 2 \cdot \frac{\dt \zeta(t)_i}{\dt t^+} = -\AtA_i \zeta(t) + \Atb_i + 1.$$
Let now $-\AtA_i \zeta(t) + \Atb_i \in  [-1, 1]$. Then, 
$$ \frac{\dt \zeta(t)_i}{\dt t^+} = 0.$$
Assume it is not and $\frac{\dt \zeta(t)_i}{\dt t^+} > 0$. Then, for all $h>0$ sufficiently small, $\zeta_i(t+h)>0$ but  $\frac{\dt \zeta(t+h)_i}{\dt t^+} < 0$, which is a contradiction that can also been shown if $\frac{\dt \zeta(t)_i}{\dt t^+} < 0$. This concludes the proof.
\end{proof}

The semi-derivatives in Proposition \ref{Prop:ODEs} and the discussion in the beginning of Section~\ref{Sec_Fell} allow us to define the infinitesimal generators of $(\zeta(t))_{t \geq 0}$, $(\zeta^\dt(t))_{t \geq 0}$, and $(\zeta^\rg(t))_{t \geq 0}$. In this order, they  are given by
\begin{align*}
   \mathcal{A}_\zeta f= \left\langle \nabla f, -\nabla \frac12 \Phi_\dt + \frac12 G \right\rangle,  \quad 
   \mathcal{A}^\dt_\zeta f = \left\langle \nabla  f, -\nabla \Phi_\dt  \right\rangle,  \quad
    \mathcal{A}^\rg_\zeta f = \left\langle \nabla  f,- \mathrm{sgn}  \right\rangle,  \qquad  (f \in C^1_0(X)).
\end{align*}

Using the infinitesimal generators of $(\zeta^\dt(t))_{t \geq 0}$, and $(\zeta^\rg(t))_{t \geq 0}$, we can now discuss the stochastic approximation. Note that $(\bsi(t),\theta(t))_{t \geq 0}$ is a simple example for a piecewise-deterministic Markov process in the sense of Davis \cite{Davis1984};  see also \cite{Benaim2012_quant,Cloez2015}.  \review{Thus, the  construction of the generators can, for instance, be found similarly in those articles.} We summarise:
\begin{proposition}
$(\bsi(t),\theta(t))_{t \geq 0}$ is a Markov process and Feller with infinitesimal generator
\begin{align*}
      \mathcal{A}_\theta f(x,i) &= \begin{cases} 
      \langle \nabla_{x}f(x,0), -\nabla  \Phi_\dt(x) \rangle + \lambda(f(x,1)-f(x,0)) ,  &\text{ if } i = 0, \\
      \langle \nabla_{x}f(x,1),- \mathrm{sgn}(x)\rangle + \lambda(f(x,0)-f(x,1)) ,  &\text{ if } i = 1, 
      \end{cases}
\end{align*}
for $f(\cdot,0), f(\cdot, 1) \in C^1_0(X)$ and  $x \in X$.
\end{proposition}

\subsection{Classification} \label{subsec_Fell_class}
We now discuss the right semi-derivatives of the processes $(\xi(t))_{t \geq 0}$ and $(\xi^\dw(t))_{t \geq 0}$.
\begin{proposition} \label{Prop:ODEs_class}
The functions $(\xi(t))_{t \geq 0}$ and $(\xi^\dw(t))_{t \geq 0}$ are right semi-differentiable and have the following semi-derivatives:
\begin{align*}
  \frac{\dt \xi(t)}{\dt t^+} &:= -\frac12 \nabla  \lPhi(\xi(t)) + \frac12 G'(\xi(t)),  \\
   \frac{\dt \xi^\dw_i(t)}{\dt t^+}  &:= G_i''(\xi^\dw_i(t)), 
   \end{align*}
   for $t \geq 0$, where for $\xi \in X$ and  $i \in \{1,\ldots,n\}$, we have
   \begin{align*}
  G_i'(\xi) &:=  \begin{cases} \varepsilon^{-1}, &\text{ if } \xi_i < -1, \\ 
  \varepsilon^{-1}\xi_i, &\text{ if } \xi_i \in (-1,1),\\
   -\varepsilon^{-1}, &\text{ if } \xi_i > 1, \\
   (\nabla\lPhi(\xi))_i, &\text{ if } \xi_i \in \{-1,1\}, (-\nabla \lPhi(\xi))_i \in [-\varepsilon^{-1},\varepsilon^{-1}], \\
   -\varepsilon^{-1}, &\text{ if } \xi_i \in \{-1,1\}, (-\nabla \lPhi(\xi))_i > \varepsilon^{-1}, \\
   \varepsilon^{-1}, &\text{ if } \xi_i \in \{-1,1\}, (-\nabla \lPhi(\xi))_i < -\varepsilon^{-1},
   \end{cases}\\
   G_i''(\xi) &:= \begin{cases} 1, &\text{ if } \xi_i < -1, \\ 
 \xi_i, &\text{ if } \xi_i \in (-1,1),\\
   -1, &\text{ if } \xi_i > 1, \\
   0 &\text{ if } \xi_i \in \{-1,1\}.
   \end{cases}
\end{align*}
\end{proposition}
\begin{proof}
The proof runs along the same lines as that of Proposition~\ref{Prop:ODEs}. 
\end{proof}
Given these semi-derivative and the discussion about ODEs in the beginning of Section~\ref{Sec_Fell}, we know that the infinitesimal generators of $(\xi(t))_{t \geq 0}$, $(\xi^\lin(t))_{t \geq 0}$, $(\xi^\dw(t))_{t \geq 0}$ are given by
\begin{align*}
   \mathcal{A}_\xi f = \left\langle \nabla f, -\nabla \frac12 \lPhi + \frac12 G'\right\rangle, \quad  \mathcal{A}^\lin_\xi f = \left\langle \nabla  f, -\nabla  \lPhi  \right\rangle, \quad 
    \mathcal{A}^\dw_\xi f = \left\langle \nabla  f,G_i''  \right\rangle \qquad (f \in C^1_0(X)).
\end{align*}
Again, we can follow \cite{Davis1984} to show that $(\bsi(t), \eta(t))_{t \geq 0}$ is Feller and to obtain its infinitesimal generator:
\begin{proposition}
$(\bsi(t),\eta(t))_{t \geq 0}$ is a Markov process and Feller with infinitesimal generator
\begin{align*}
      \mathcal{A}_\eta f(x,i) &= \begin{cases} 
      \langle \nabla_{x}f(x,0), -\nabla  \lPhi(x) \rangle + \lambda(f(x,1)-f(x,0)) ,  &\text{ if } i = 0, \\
      \langle \nabla_{x}f(x,1),G''(x)\rangle + \lambda(f(x,0)-f(x,1)) ,  &\text{ if } i = 1, 
      \end{cases}
\end{align*}
for $f(\cdot,0), f(\cdot, 1) \in C^1_0(X)$ and  $x \in X$.
\end{proposition}

\section{Longtime behaviour} \label{sec_longtime}
We now study the longtime behaviour \review{and convergence} of the subgradient flows $(\zeta(t))_{t \geq 0}$, $(\xi(t))_{t \geq 0}$ and their stochastic approximations $(\theta(t))_{t \geq 0}$, $(\eta(t))_{t \geq 0}$.  Indeed, we see that the subgradient flows converge to a point. In general, the stochastic approximations do not converge to a single point. Instead, we study their convergence to potentially existing stationary probability distributions.

The longtime behaviour of subgradient flows of this form has been studied by, e.g., \cite{BRUCK1975,Marcellin2006}. From them, we obtain the following result:
\begin{proposition} \label{prop_det_conv}
\begin{itemize}
    \item[(i)] The sparse inversion flow $(\zeta(t))_{t \geq 0}$ converges to the stationary point of $\hPhi$, i.e. $\zeta(t) \rightarrow \theta_*$, as $t \rightarrow \infty$.
    \item[(ii)] The discrete Allen\review{--}Cahn equation $(\xi(t))_{t \geq 0}$ converges to the minimal point of its trajectory, i.e. $\tilde{\Phi}(\xi(t)) \rightarrow \inf_{T > 0}\tilde{\Phi}(\xi(T))$, as $t \rightarrow \infty$.
\end{itemize}

\end{proposition}
\begin{proof}
(i) is a classical result by Bruck \cite{BRUCK1975}, where it is stated in Theorem 4. For (ii), we employ Proposition 4.1 from \cite{Marcellin2006}, which holds as $\tilde{\Phi}$ is bounded below and primal lower nice, see Proposition~\ref{prop_AC_exist}.
\end{proof}
In the sparse inversion problem, we converge to the unique minimiser of $\hPhi$. In the classification problem, this is not so clear. Our proposition only states that the discrete Allen\review{--}Cahn equation to which we apply the target function finds the infimum of the target function over all visited states. We should note that, in general, $\tilde \Phi$ has no unique minimiser. Consider, for instance, the case, where $X' = X$, $d = 0$, and $\PtP - \varepsilon\triangle' = \frac{1}{n}\mathrm{Id}_X$. Then, the vectors $(-1,\ldots,-1)^T$ and $(1,\ldots,1)^T$ are both minimisers of $\tilde{\Phi}$. We should also note that the dynamical system $(\zeta(t))_{t \geq 0}$ does not need to get close to a minimiser. In the aforementioned setting, $(0,\ldots,0)^T$ is a stationary point of the dynamical system. A process starting from there would not converge to a minimiser of the target function.

We now move on to the stochastic approximations. As mentioned before, we usually cannot expect convergence to a single stationary point, We discuss the convergence of the stochastic approximations in terms of the Wasserstein distance. It is defined as follows: Let $\pi, \pi' \in \mathrm{Prob}(X)$ be two probability measures on $(X, \mathcal{B}X)$ and $\mathrm{Coup}(\pi, \pi') \subseteq \mathrm{Prob}(X \times X)$ be the set of couplings of $\pi, \pi'$. Let $p \in (0, 1)$. Then, the \emph{Wasserstein distance}  of $\pi, \pi'$ is defined as $\mathrm{W}_p(\pi, \pi')$, where
$$
\mathrm{W}_p(\pi, \pi') := \inf_{\Gamma \in \mathrm{Coup}(\pi, \pi')} \int_{X \times X} \min\{1, \|\zeta - \zeta'\|^p_2\} \Gamma(\mathrm{d}\zeta, \mathrm{d}\zeta').
$$
Convergence in this Wasserstein distance is equivalent to weak convergence of the probability measures, see, e.g. \cite{Villani}.
We start with the stochastic approximation of the sparse inversion flow in the next section and then discuss the discrete Allen--Cahn equation after that.

\subsection{Sparse inversion} \label{subsec_lontime_SP}
To understand the asymptotic behaviour of $(\theta(t))_{t \geq 0}$, we need to have a look at the behaviour of the subflows $(\zeta^\dt(t))_{t \geq 0}$ and $(\zeta^\rg(t))_{t \geq 0}$.
Both $(\zeta^\dt(t))_{t \geq 0}$ and $(\zeta^\rg(t))_{t \geq 0}$ converge to a stationary point. Indeed, $(\zeta^\rg(t))_{t \geq 0}$ converges in finite time, $(\zeta^\dt(t))_{t \geq 0}$ converges exponentially. Moreover, we show boundedness of $(\theta(t))_{t \geq 0}$.
\begin{lemma} \label{lemma_simple_conver} Let $\zeta^{(0)}, \zeta^{(1)} \in X$ be initial values and let $\sigma_{\min}$ be the minimal singular value of $A$. Then, 
\begin{enumerate}
    \item[(i)] $
\|\zeta^\rg(t) \|_2 \rightarrow 0$, as $t \rightarrow \infty$, 
indeed, $\zeta^\rg(t) = 0$, if $t \geq \max_{i=1}^n |\zeta^{(0)}_i|$. Also, $\|\zeta^{\rg}(t,\zeta^{(0)}) - \zeta^{\rg}(t,\zeta^{(1)})  \|_2 \leq \|\zeta^{(0)} - \zeta^{(1)}  \|_2$.
\item[(ii)] $\|\zeta^{\dt}(t,\zeta^{(0)}) - \zeta^{\dt}(t,\zeta^{(1)})  \|_2 \leq \exp(-\sigma_{\min}^2t) \|\zeta^{(0)} - \zeta^{(1)}  \|_2 \qquad (t \geq 0)$,
\item[(iii)] The process $(\theta(t))_{t \geq 0}$ is bounded whenever the initial value $\theta(0) = \theta^{(0)} \in X$ is bounded.
\end{enumerate}
\end{lemma}
\begin{proof}
(i) follows immediately from the definition of the subgradient flow. (ii) is given by the following simple calculation:
\begin{align*}
    \|\zeta^{\dt}(t,\zeta^{(0)}) - \zeta^{\dt}(t,\zeta^{(1)})  \|_2 &= \| \exp(-t\AtA)(\zeta^{(0)}-\zeta^{(1)}) \|_2 \\ &\leq \| \exp(-t\AtA)\|_2 \|\zeta^{(0)}-\zeta^{(1)} \|_2 = \exp(-\sigma_{\min}^2t)\|\zeta^{(0)}-\zeta^{(1)} \|_2.
\end{align*}
(iii) We can show boundedness for the two subprocesses in terms of the respective initial value:
\begin{align*}
 \|\zeta^{\dt}(t,\zeta^{(0)}) \|_2-\|\AtA^{-1}\Atb\|_2 &\leq \|\zeta^{\dt}(t,\zeta^{(0)}) - \AtA^{-1}\Atb\|_2 \\ &\leq \exp(-\sigma_{\min}^2t)\| \zeta^{(0)} - \AtA^{-1}\Atb\|_2  
\end{align*}
using the reversed triangular inequality.
Moreover,
$\|\zeta^{\rg}(t,\zeta^{(0)})  \|_2 \leq \|\zeta^{(0)}  \|_2$.
Thus, the process $(\theta(t))_{t \geq 0}$ is bounded as well.
\end{proof}

With this result, we now have all the ingredients to study the asymptotic behaviour of $(\theta(t))_{t \geq 0}$.

\begin{theorem} \label{theo_SP_app_longtime}
The stochastic approximation $(\theta(t))_{t \geq 0}$ has a unique stationary measure $\pi_* \in \mathrm{Prob}(X)$, for any initial value $\theta(0) = \theta^{(0)} \in X$ and any $p \in (0,1)$, there are constants $c, c' > 0$ such that
$$
\mathrm{W}_p(\pi_*,\mathbb{P}(\theta(t)\in \cdot)) \leq c \exp(-c't) \qquad (t \geq 0).
$$
\end{theorem}
\begin{proof}
The convergence follows from Theorem 1.4 in \cite{Cloez2015}, which requires the results of Lemma~\ref{lemma_simple_conver} and the stochastic approximation to be Feller. Uniqueness follows from $(\theta(t))_{t \geq 0}$  contracting in the Wasserstein distance itself and the Banach Fixed Point Theorem.
\end{proof}
Hence, we see that the algorithm converges exponentially to a stationary measure. We refer to this stationary measure as the \emph{implicit regularisation}. It determines how the stochasticity in the algorithm regularises the optimisation problem and is especially important in non-convex optimisation, see, e.g., \cite{smith2021on}. Moreover, the probability distribution is sometimes used for uncertainty quantification, see, \cite{Mandt2017}.

It is noteworthy in this case, that the probability distribution $\pi_*$ is often not a single point -- in contrast to the discrete forward-backward splitting given in Algorithm~\ref{algo_fwbwS} with sufficiently small step sizes. Thus, either the randomisation of the waiting times prohibits the contraction to the minimiser of \eqref{Eq:Optprob} or the contraction in the algorithm is a numerical artefact (which seems more likely).

It is fairly vital for the proof that the operator $A$ is of full rank. While this is true for many sparse regression problems, it is clearly not true for all usual inverse problems (like inpainting or sparse dictionary learning). Moreover,  we would usually assume that the regulariser improves the convergence behaviour. One way to solve this problem is $\|\cdot\|_2^2 + \|\cdot\|_1$ regularisation, which is not untypical in practice, see, e.g., \cite{Stadler2007}.

\review{Finally, we note that the constant $c$ depends on the distance of the process' initial value $(\theta^{(0)},i(0))$ to the stationary measure $\pi_* \otimes \mathrm{Unif}\{0,1\}.$ Moreover, the constant $c'$ depends on $\lambda$ and $\sigma_{\min}^2$. More details are given in \cite{Benaim2012_quant,Cloez2015}.}

\subsection{Classification}
As in the sparse inversion case, we first discuss the asymptotic properties of the subflows $(\xi^{\lin}(t))_{t \geq 0}$ and $(\xi^{\dw}(t))_{t \geq 0}$, as well as the boundedness of the stochastic approximation $(\eta(t))_{t \geq 0}$.

\begin{lemma} \label{lemma_simple_conver_AC}
Let $\xi^{(0)}, \xi^{(1)} \in X$ be initial values and let $\mu_{\min} > 0$ be the minimal eigenvalue of $\PtP - \varepsilon\triangle'$. Then, for $t \geq 0$, we have
\begin{itemize}
    \item[(i)] $\|\xi^{\dw}(t, \xi^{(0)})-\xi^{\dw}(t, \xi^{(1)})\|_2 \leq \exp(\varepsilon^{-1}t)\|\xi^{(0)}-\xi^{(1)}\|_2$,
    \item[(ii) ]there is a time $t' \geq 0$ such that for all $t \geq t'$, we have $\|\xi^{\dw}(t, \xi^{(0)})-\xi^{\dw}(t, \xi^{(1)})\|_2 \leq 2 \sqrt{n}$,
       \item[(iii)] $\|\xi^{\dt}(t,\xi^{(0)}) - \xi^{\dt}(t,\xi^{(1)})  \|_2 \leq \exp(-\mu_{\min}t) \|\xi^{(0)} - \xi^{(1)}  \|_2$,
    \item[(iv)] The stochastic approximation $(\eta(t))_{t \geq 0}$ is bounded. 
\end{itemize}
\end{lemma}
\begin{proof}
The dynamical system $(\xi^{\dw}(t))_{t \geq 0}$ in (i) behaves differently for initial values in different areas of $X$. As the process has the same convergence behaviour in each component, we assume without loss of generality that $X = \mathbb{R}$. In the intervals $(-\infty, -1)$ and $(1, \infty)$, the process contracts toward $-1$ and $1$, respectively. ${-1, 0, 1}$ are stationary points. Within the intervals $(-1,0)$ and $(0,1)$ the process expands and converges toward $-1$ and $1$ respectively. Where the process expands, it expands like the ODE $\frac{\mathrm{d}x}{\mathrm{d}t} = \varepsilon^{-1}x$, by definition. Thus, we obtain the bound in (i). (ii) follows from the fact that the components of $(\xi^{\dw}(t))_{t \geq 0}$ converge to $\{-1, 0, 1\}$ in finite time.
The proof of (iii) is identical to Lemma~\ref{lemma_simple_conver}(ii). (iv) is implied by (ii) and (iii).
\end{proof}
Note that we show in (i) not a contraction of the dynamical system, but rather  a bound on its expansion. In (ii), we show additionally, that after some finite time this expansion is bounded by a constant that only depends on the dimension of $X$. Given these statements, we can show the following theorem. The proof is identical to that of Theorem~\ref{theo_SP_app_longtime}.

\begin{theorem} \label{theo_AC_app_longtime} Let $\mu_{\min} > \varepsilon^{-1}.$
Then, the stochastic approximation $(\eta(t))_{t \geq 0}$ has a unique stationary measure $\pi_\dagger \in \mathrm{Prob}(X)$, for any initial value $\eta(0) = \eta^{(0)} \in X$ and any $p \in (0,1)$, there are constants $c, c' > 0$ such that
$$
\mathrm{W}_p(\pi_\dagger,\mathbb{P}(\eta(t)\in \cdot)) \leq c \exp(-c't) \qquad (t \geq 0).
$$
\end{theorem}

As for the discrete Allen--Cahn equation $(\xi(t))_{t \geq 0}$, convergence to a unique stationary measure is not automatic for the stochastic approximation $(\eta(t))_{t \geq 0}$. It actually depends on the choice of the discrete Laplacian and $\varepsilon$. Note that $\PtP$ is usually not of full rank, which is why the smallest eigenvalue of $-\triangle'$ is likely a good approximation to $\mu_{\min}/\varepsilon$. In case $\varepsilon \approx 0$, we cannot expect the inequality in Theorem~\ref{theo_AC_app_longtime} to hold.

\section{Approximation properties} \label{sec_perturbed}
In this last theoretical section of this work, we study the approximation properties of the stochastic approximations. In particular, we are looking for result\review{s} of type $$(\theta(t))_{t \geq 0} \Rightarrow (\zeta(t))_{t \geq 0}, \qquad \qquad (\eta(t))_{t \geq 0} \Rightarrow (\xi(t))_{t \geq 0}, $$
as $\lambda \rightarrow \infty$. Here, `$\Rightarrow$' means convergence in the weak sense probabilistically and uniformly in time $t\geq 0$.
Indeed, we aim to show that the stochastic approximations can approximate the deterministic dynamics at any accuracy. Results of this type have been discussed in \cite{Jin2021, Latz2021}
considering smooth stochastic optimisation and in \cite{Dupuis} considering Markov chain Monte Carlo. The theoretical foundation is given by Kushner's perturbed test function theory \cite{Kushner1984,Kushner1990}. Indeed, we aim at employing Theorem 2 in Chapter 3 of \cite{Kushner1984}. Intuitively, this theorem states, that we can show weak convergence of one family of Feller processes to another Feller process, if the family is tight and if the corresponding infinitesimal generators converge in a certain sense. Here, we are allowed to test the family of generators with perturbed test functions.
 
 We start with the tightness, which follows from uniform equicontinuity of the processes.
 \begin{lemma} \label{lemma_tight} Let $\theta^{(0)}, \eta^{(0)} \in X$. 
 $(\bsi(t), \theta(t))_{t \geq 0}$ and   $(\bsi(t), \eta(t))_{t \geq 0}$ are  tight with respect to $\lambda > 0$.
 \end{lemma}
 \begin{proof} $(\bsi(t))_{t \geq 0}$ is certainly tight, as it lives on a finite state space.
 We show tightness of $(\theta(t))_{t \geq 0}$ by proving that the paths of $(\theta(t))_{t \geq 0}$ are uniformly equicontinuous with respect to $\lambda$ and bounded. Boundedness is implied by the boundedness of $\theta_0$ and the contractivity of the piecewise processes, see Lemma~\ref{lemma_simple_conver}.  Note that --  independently of $\lambda > 0$ -- the paths $(\theta(t))_{t \geq 0}$ are functions constructed piecewise from Lipschitz continuous functions $(\zeta^{\dt}(t))_{t  \geq 0}$ and $(\zeta^{\rg}(t))_{t  \geq 0}$. Indeed, note that $(\zeta^{\dt}(t))_{t  \geq 0}$ is Lipschitz continuous due to its differentiability and contractivity. A function that is constructed piecewise from Lipschitz continuous functions is Lipschitz continuous itself, see Proposition 4.1.2 in \cite{Scholtes2012}. Here, the Lipschitz constant does not depend on $\lambda$. Hence, the paths are uniformly equicontinuous and, thus, tight; see Theorem 4 in Chapter 2 of \cite{Kushner1984}. $(\eta(t))_{t \geq 0}$ satisfies the same boundedness and smoothness conditions as $(\theta(t))_{t \geq 0}$. Thus, $(\bsi(t), \eta(t))_{t \geq 0}$ is tight, as well.
 \end{proof}
 
Now, we need to define the test functions and perturbed test functions.
 First, we focus on the sparse inversion case.
 To this end, let $f \in C^2_0(X)$ be a test function.
 We define the associated perturbed test function by $f^\lambda(t) = f(\theta(t))$ $(t \geq 0)$ and the action of $\mathcal{A}_\theta$ on $f^\lambda$ by
 \begin{align*}
      \mathcal{A}_\theta f^\lambda(t) &= \begin{cases} 
      \langle \nabla_{x}f(\theta(t)), -\nabla  \Phi_\dt(\theta(t)) \rangle,  &\text{ if } \bsi(t) = 0, \\
      \langle \nabla_{x}f(\theta(t)),- \mathrm{sgn}(\theta(t))\rangle  ,  &\text{ if } \bsi(t) = 1
      \end{cases} \qquad \qquad (t \geq 0).
\end{align*}
The following lemma is the technical main result to show the weak convergence motivated above.
 \begin{lemma} \label{lemma_pertu_Spar} Let $T < \infty$. Then,
 \begin{align}
 \label{eq:p-lim1}
    \sup_{\lambda>0, t \geq 0} &\mathbb{E}\left[\|f^\lambda(t)\|_2 \right]< \infty,  \\ \lim_{\lambda \rightarrow \infty}&\mathbb{E}\left[\|f^\lambda(t) - f(\theta(t))\|_2 \right] = 0 \qquad (t  \geq 0). \label{eq:p-lim2}
 \end{align}
 and
 \begin{equation} \label{eq:cond_exp}
\mathbb{E} \left\lvert\left[\int_{t}^T \mathbb{E}\left[\mathcal{A}_\theta f^\lambda(u) - \mathcal{A}_\zeta f(\theta(u))\mid \bsi(s) : t \geq s\right] \mathrm{d}u \right\rvert\right] \rightarrow 0 \qquad (\lambda \rightarrow \infty),
 \end{equation}

 for any $t \leq T$.
 \end{lemma}
 \begin{proof}
 1. The boundedness in \eqref{eq:p-lim1} follows from the boundedness of the intial value, the contractivity of the process, and the boundedness of $f$. \eqref{eq:p-lim2} is trivial, as by definition $f^\lambda = f(\theta(\cdot))$.
 
2. Let $u \in [t, T]$. Note that $(\bsi(t))_{t \geq 0}$ is a Markov process.  Thus, 
$$
\mathbb{E}\left[\mathcal{A}_\theta f^\lambda(u) - \mathcal{A}_\zeta f(\theta(u))\mid \bsi(s) : t \geq s\right] = \mathbb{E}\left[\mathcal{A}_\theta f^\lambda(u) - \mathcal{A}_\zeta f(\theta(u))\mid \bsi(t) \right]
$$
Let now $i \in I$. Then, 
\begin{align*}
    \mathbb{E}&\left[\mathcal{A}_\theta f^\lambda(u) - \mathcal{A}_\zeta f(\theta(u))\mid \bsi(t) = i\right] \\ &= I_t(\{ 0 \} \mid i) \langle \nabla_{\theta} f^\lambda(u),- \left( \mathbf{1}[\theta_i(u) > 0] - \mathbf{1}[\theta_i(u) < 0] \right)_{i=1}^n\rangle + I_t(\{ 1 \} \mid i) \langle \nabla_{\theta}f^\lambda(u), -\nabla  \Phi_\dt(\theta(t)) \rangle \\ &\qquad \qquad - \langle \nabla f(\theta(t)), -\nabla \frac12 \Phi_\dt(x) + \frac12 G(x)\rangle, \\
    &= \left(\frac{- \exp(-2\lambda (u-t))}{2} + \exp(- 2\lambda  (u-t))\mathbf{1}[i = 0] \right) \langle \nabla_{\theta} f^\lambda(u),- \left( \mathbf{1}[\theta_i(u) > 0] - \mathbf{1}[\theta_i(u) < 0] \right)_{i=1}^n\rangle \\
    &\qquad \qquad + \left(\frac{- \exp(-2\lambda (u-t))}{2} + \exp(- 2\lambda  (u-t))\mathbf{1}[i = 1] \right) \langle \nabla_{\theta}f^\lambda(u), -\nabla  \Phi_\dt(\theta(u)) \rangle \\
    &= \exp(-2\lambda (u-t))\left( \mathbf{1}[i = 0]-\frac{ 1}{2} \right) \langle \nabla_{\theta} f^\lambda(u),- \left( \mathbf{1}[\theta_i(u) > 0] - \mathbf{1}[\theta_i(u) < 0] \right)_{i=1}^n\rangle \\
    &\qquad \qquad + \exp(-2\lambda (u-t)) \left(\mathbf{1}[i = 1]-\frac{1}{2} \right) \langle \nabla_{\theta}f^\lambda(u), -\nabla  \Phi_\dt(\theta(u)) \rangle \\
    &=: \exp(-2\lambda (u-t)) H(u,i).
\end{align*}
where the second  equality holds for $\mathrm{Unif}[t, T]$-almost every $u$, due to the absolute continuity of $(\theta(u))_{u \in [t, T]}.$
Now, we have 
\begin{align*}
    \left\lvert \int_t^T\exp(-2\lambda (u-t)) H(u,i) \mathrm{d}u \right\rvert^2 &\leq \int_t^T H(u,i)^2 \mathrm{d}u \int_t^T\exp(-4\lambda (u-t))  \mathrm{d}u \\
    &= \int_t^T H(u,i)^2 \mathrm{d}u \cdot \exp(4\lambda t) \cdot \frac{\exp(-4\lambda t)-\exp(-4 \lambda T)}{4 \lambda} \\
    &= \int_t^T H(u,i)^2 \mathrm{d}u \cdot  \frac{1-\exp(-4 \lambda (T-t))}{4 \lambda}
\end{align*}
Finally, taking expectation:
\begin{align*}
    \mathbb{E}\left[ \left\lvert \int_t^T\exp(-2\lambda (u-t)) H(u,i) \mathrm{d}u \right\rvert^2 \right] \leq \mathbb{E}\left[\int_t^T H(u,i)^2 \mathrm{d}u \right] \cdot  \frac{1-\exp(-4 \lambda (T-t))}{4 \lambda},
\end{align*}
the right-hand side of which goes to $0$ as $\lambda \rightarrow \infty$.
 \end{proof}
 
Second, we discuss the same result for the stochastic approximation of the discrete Allen--Cahn flow. We again choose $f \in C^2_0(X)$ and $f^\lambda(t) :=  f(\eta(t))$ $(t \geq 0)$. The action of $\mathcal{A}_\eta$ on $f^\lambda$ is given by
\begin{align*}
      \mathcal{A}_\eta f^\lambda(t) &= \begin{cases} 
      \langle \nabla_{x}f(\eta(t)), -\nabla  \lPhi(\eta(t)) \rangle,  &\text{ if } i = 0, \\
      \langle \nabla_{x}f(\eta(t)),G''(\eta(t))\rangle,  &\text{ if } i = 1, 
      \end{cases} \qquad \qquad (t \geq 0).
\end{align*}
As for the sparse inversion problem, we can show the following statement about Allen--Cahn.
\begin{lemma} \label{lemma_pertu_AC} Let $T < \infty$. Then,
 \begin{align*}
    \sup_{\lambda>0, t \geq 0} &\mathbb{E}\left[\|f^\lambda(t)\|_2 \right]< \infty,  \\ \lim_{\lambda \rightarrow \infty}&\mathbb{E}\left[\|f^\lambda(t) - f(\eta(t))\|_2 \right] = 0 \qquad (t  \geq 0).
 \end{align*}
 and
 \begin{equation*}
\mathbb{E} \left\lvert\left[\int_{t}^T \mathbb{E}\left[\mathcal{A}_\eta f^\lambda(u) - \mathcal{A}_\xi f(\eta(u))\mid \bsi(s) : t \geq s\right] \mathrm{d}u \right\rvert\right] \rightarrow 0 \qquad (\lambda \rightarrow \infty),
 \end{equation*}
 for any $t \leq T$.
 \end{lemma}
\begin{proof}
 The proof proceeds identical to that of Lemma~\ref{lemma_pertu_Spar}.
\end{proof}
 
We now formulate the main approximation result. 
 \begin{theorem} \label{Theo_weakConv}
 Let $\zeta(0) = \theta(0) = \theta^{(0)} \in X$ and $\xi(0) = \eta(0) = \eta^{(0)} \in X$. Then, $(\theta(t))_{t \geq 0} \Rightarrow (\zeta(t))_{t \geq 0},$ and $(\eta(t))_{t \geq 0} \Rightarrow (\xi(t))_{t \geq 0},$
as $\lambda \rightarrow \infty$.
 \end{theorem}
 \begin{proof}
 It follows from Theorem 4 in Chapter 2 of \cite{Kushner1984} due to Lemmas~\ref{lemma_tight}, \ref{lemma_pertu_Spar}, and \ref{lemma_pertu_AC}.
 \end{proof}

\section{Illustrations} \label{Sec_Illust}
In the following, we aim to illustrate the use of the stochastic approximations of the sparse inversion flow and the discrete Allen--Cahn equation. We \review{start with} simplified examples allowing us \review{to} stay close to the discussed setting. Thus, we solve all ODEs and differential inclusions by their analytical solution  \review{in Subsections~\ref{subsec:simple_exmpl} and \ref{Subsec_Ill_class1}. In Subsection~\ref{Subsec_Ill_class2}, we then move on to a higher dimensional classification problem. Here, we solve the linear ODEs with a stable time stepping scheme.}
\subsection{Sparse inversion} \label{subsec:simple_exmpl}
We consider again the example we use as an illustration in Figure~\ref{fig:simple_example}. Here, we minimise the target function $\hPhi(\theta) := \frac{1}{2} (\theta-4)^2 + |\theta|$ which is defined for $\theta \in X := \mathbb{R}$. 

The minimiser can be computed analytically in this case, giving $\theta_* = 3.$ Our stochastic approximation uses, of course, the potentials:
$$
\dPhi(\theta) := \frac{1}{2} (\theta-4)^2,  \qquad \rPhi(\theta) :=  |\theta| \qquad (\theta \in X).
$$
We now simulate $(\theta(t))_{t \geq 0}$ in this case for $\lambda \in \{0.25, 2.5, 25, 250\}$. We generate a total of $10^4$ realisations of these processes starting at $\theta^{(0)} = 0$ and keep $\theta(20)$ in the memory. We use these samples to depict the probability distributions $\mathbb{P}(\theta(20) \in \cdot )$ in Figure~\ref{fig:hist_simple_exmpl}. These distributions shall act as approximations to the stationary distribution $\pi_*$ that we have discussed in Theorem~\ref{theo_SP_app_longtime}.
\begin{figure}
    \centering
    \includegraphics[scale = 0.85]{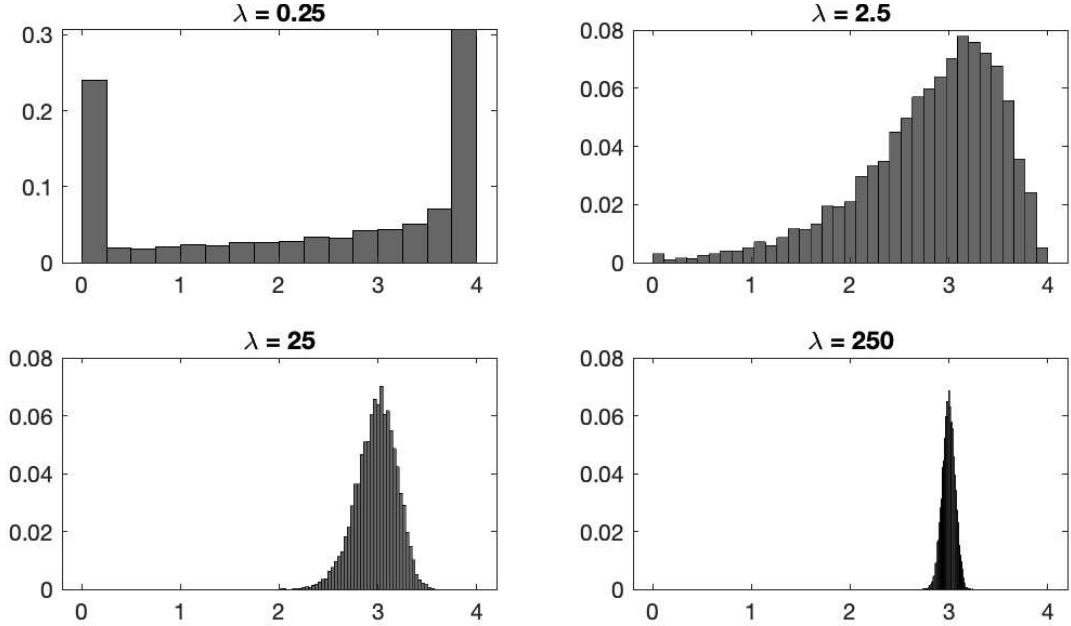}
    \caption{Histograms of $10^4$ realisations of $\theta(20)$ for $\lambda \in \{0.25, 2.5, 25, 250\}$.}
    \label{fig:hist_simple_exmpl}
\end{figure}
In the case $\lambda = 0.25$ we essentially see the process going back and forth between $0$ and $4$, which are the minimisers of $\rPhi$ and $\dPhi$, respectively. This happens due to the very large mean waiting time in $(\bsi(t))_{t \geq 0}$.  As $\lambda$ increases, the mean waiting times decrease and we see that the probability distribution $\mathbb{P}(\theta(20) \in \cdot )$ concentrates around $\theta_* = 3$, which is exactly what we expect due to Theorem~\ref{Theo_weakConv}. This can also be seen in Table~\ref{table_simple_exmpl_mean_Cov}, where we note sample mean and sample variance of the realisations. There we also see that the mean of the process is a fairly good estimate for $\theta_*$ in most cases.
\begin{table}[]
\centering
\begin{tabular}{l|rrrr}
$\lambda$          & 0.25  & 2.5   & 25    & 250   \\ \hline
Sample Mean     & 2.300  & 2.802  & 2.981  & 2.997  \\
Sample Variance & 2.556 & 0.507 & 0.041 & 0.004
\end{tabular}

\caption{Sample means and sample variances of $\mathbb{P}(\theta(20) \in \cdot)$ for different cases of $\lambda$.}
\label{table_simple_exmpl_mean_Cov}
\end{table}

\subsection{Classification in 1D} \label{Subsec_Ill_class1}
We now study a simple classification problem on the set $I = \{1,\ldots,200\}$. Indeed, we aim to find a classification vector $\eta^{\dagger} \in \{-1,1\}^{200}$ given that we observed a (small) part of the entries of $\eta^{\dagger}$ and that we assume a certain spatial clustering amongst the classes $-1$ and $1$. We give a plot of $\eta^{\dagger}$ in Figure~\ref{fig:etadagger}, where we made sure that $\eta^\dagger$ is non-trivial due to non-linear, non-monotonic, and non-oscillating behaviour.
\begin{figure}
    \centering
    \includegraphics[scale = 0.85]{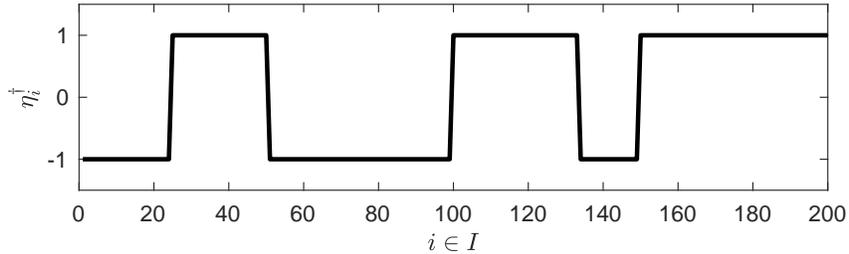}
    \caption{A plot of the underlying `true' classification $\eta^\dagger$ in the 1D example.}
    \label{fig:etadagger}
\end{figure}

To retrieve the classification based on the data, we employ the stochastic approximation of the discrete Allen--Cahn equation on $X = \mathbb{R}^{200}$. Here, we employ a very basic discrete Laplacian $\triangle'$\review{$_1$}, given by the tridiagonal matrix with diagonal values $-50$ and off-diagonal values $25$. This discrete Laplacian can be derived through a centred finite differences approximation of a Laplacian on an open interval with Dirichlet $=0$ boundary conditions and discretisation width $h = 0.2$. We set the weight $\alpha := 1$.

In our tests, we study different combinations of transition rates $\lambda \in \{1,10\}$ and $\varepsilon \in \{10^{-2}, 10^{-4}, 10^{-6}\}$. In each case, we start with $\eta(0) = 0$. We assume that we have observed every $5^{\rm th}$ entry of $\eta^\dagger$, i.e. we observe the class of $\{1, 6, \ldots, 196 \} \subseteq I$. In the following figures, we depict the stochastic approximation $(\eta(t))_{t \geq 0}$. Indeed, we plot means and standard deviations of $\eta(4), \eta(16), \eta(64)$ in Figure\review{s}~\ref{Fig:AC-2}, \ref{Fig:AC-4}, and \ref{Fig:AC-6}. Means and standard deviations are estimated through $100$ realisations of the stochastic approximation.

In the figures, we see overall that the larger $\lambda (=10)$ value helps the system to stabilise and exhibit physical behaviour  after some time. `Physical' means that the function is approximately piecewise constant with values close to $-1$ and $1$. The smaller $\lambda (=1)$ leads to a more noisy ($\varepsilon \in \{10^{-4}, 10^{-6}\}$ or completely non-physical ($\varepsilon = 10^{-2}$) behaviour. Before the system stabilises with a piecewise constant solution around $\{-1, 1\}$, it appears to be dominated by the linear process. This is expected as we start at $0$, where the double-well subgradient flow moves slowly, if at all. The combination $\varepsilon = 10^{-2}, \lambda = 10$ appears particularly efficient. The relatively large $\varepsilon$ leads to the diffusion/fidelity process to converge quickly to stationarity, which then amplifies the thresholding procedure within $[0,1]$. Computing a single realisation of $\eta(16)$ in this setting took in our experiment on a MacBook Pro (Intel i7, 4 cores @ 2.6Ghz; 16GB Ram) on average $0.22 (\pm 0.05)$ seconds. This appears to be very fast given that we solve the linear part by computing matrix exponentials.

\begin{figure}
    \centering
    \includegraphics[scale = 0.7]{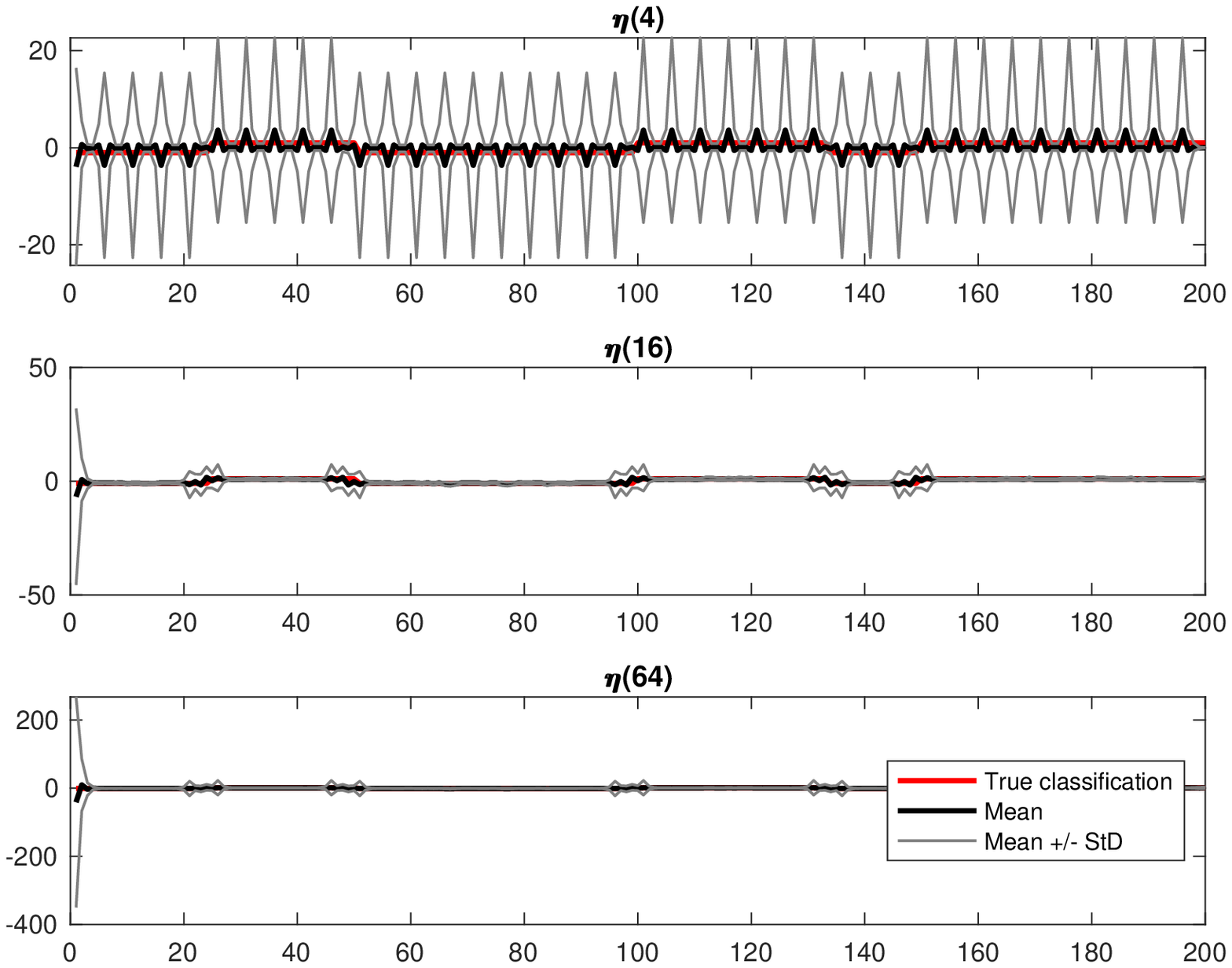} \includegraphics[scale = 0.7]{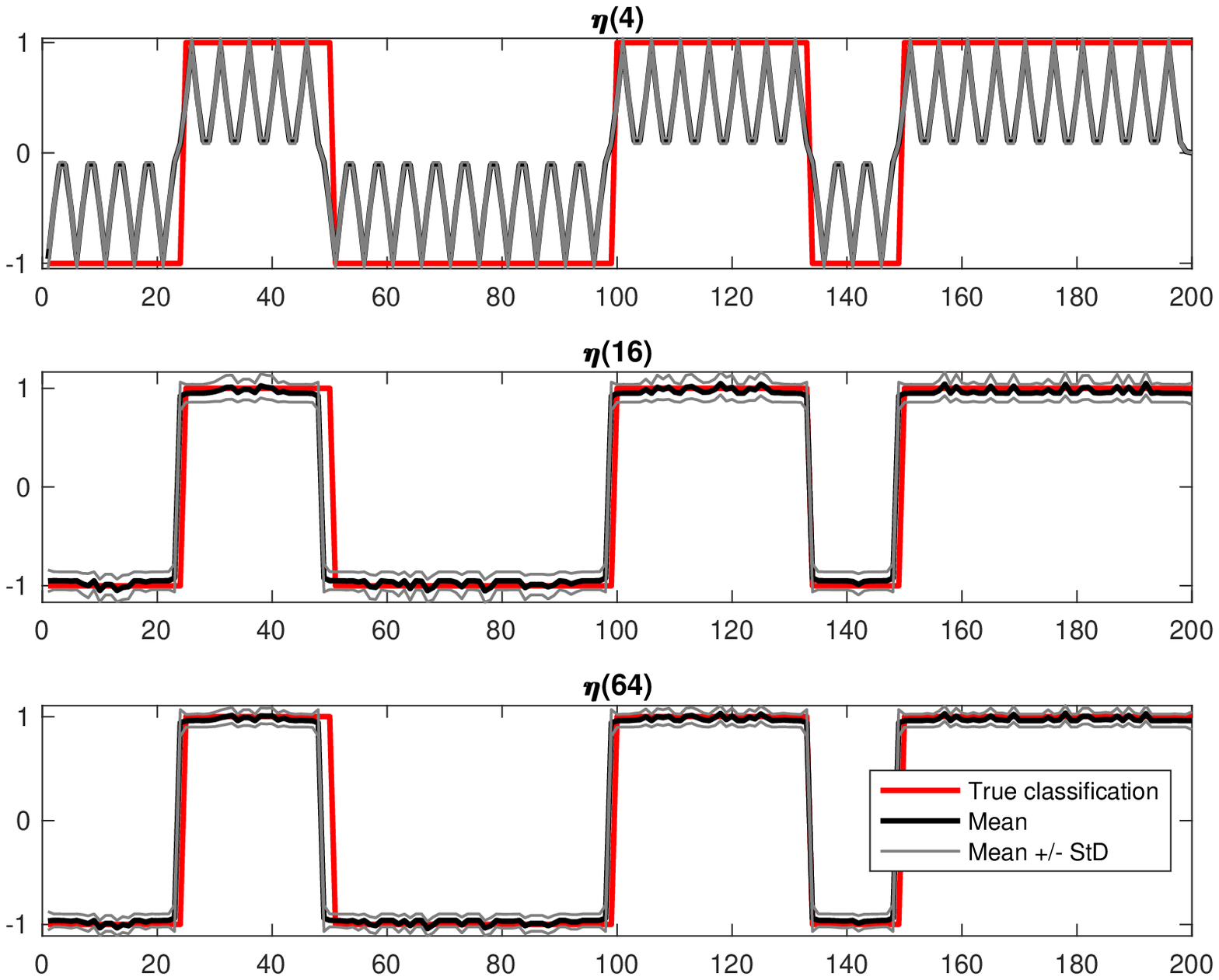}
    \caption{Plots of means (black thick lines), means $\pm$ standard deviations (grey thin lines)of $\eta(4), \eta(16), \eta(64)$, given $\varepsilon= 10^{-2}$, as well as $\lambda = 1$ (top) and $\lambda = 10$ (bottom). As a baseline, we compare with the true classification $\eta^\dagger$ (red).}
    \label{Fig:AC-2}
\end{figure}

\begin{figure}
    \centering
    \includegraphics[scale = 0.7]{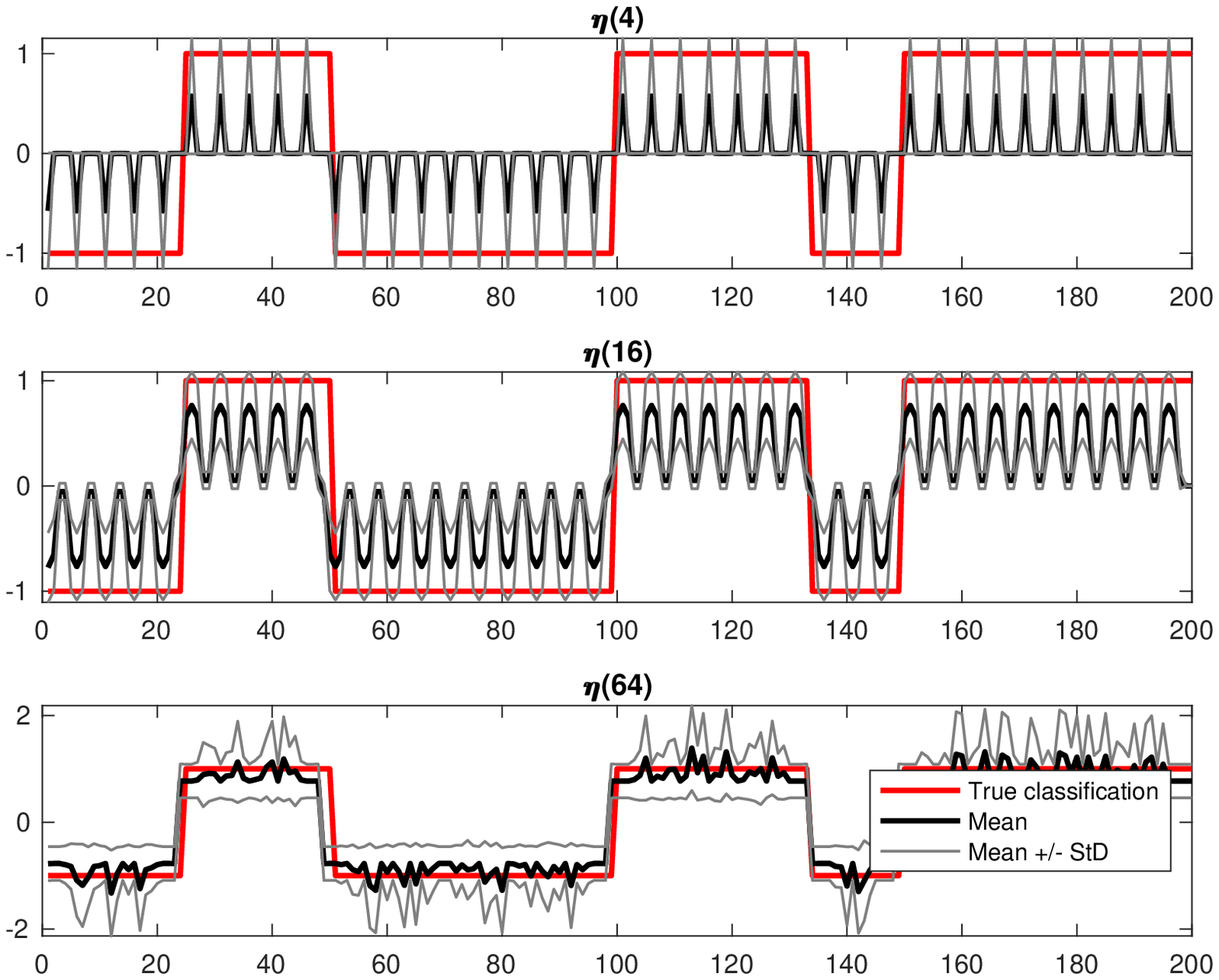} \includegraphics[scale = 0.7]{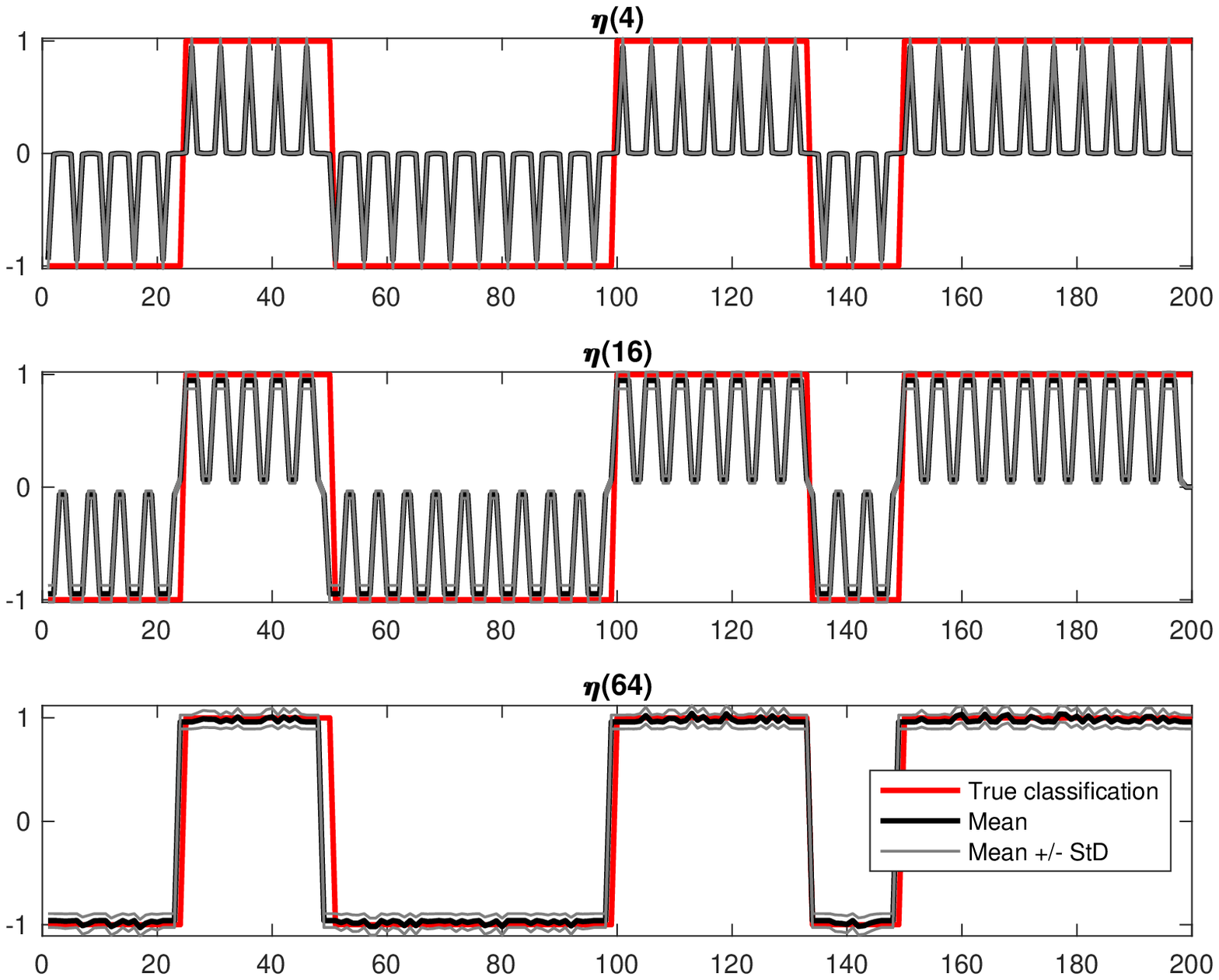}
    \caption{Plots of means (black thick lines), means $\pm$ standard deviations (grey thin lines)of $\eta(4), \eta(16), \eta(64)$, given $\varepsilon= 10^{-4}$, as well as $\lambda = 1$ (top) and $\lambda = 10$ (bottom). As a baseline, we compare with the true classification $\eta^\dagger$ (red).}
    \label{Fig:AC-4}
\end{figure}

\begin{figure}
    \centering
    \includegraphics[scale = 0.7]{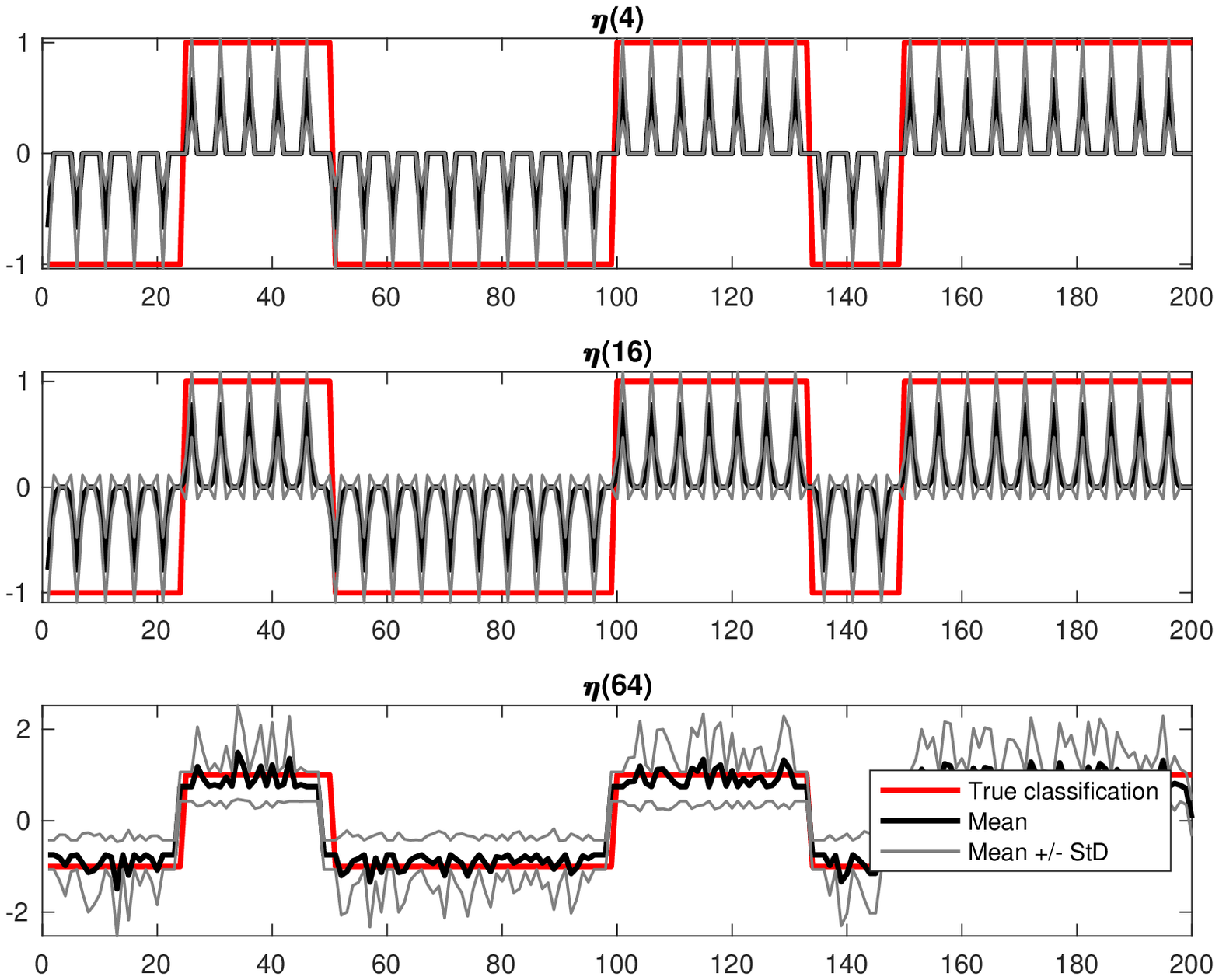} \includegraphics[scale = 0.7]{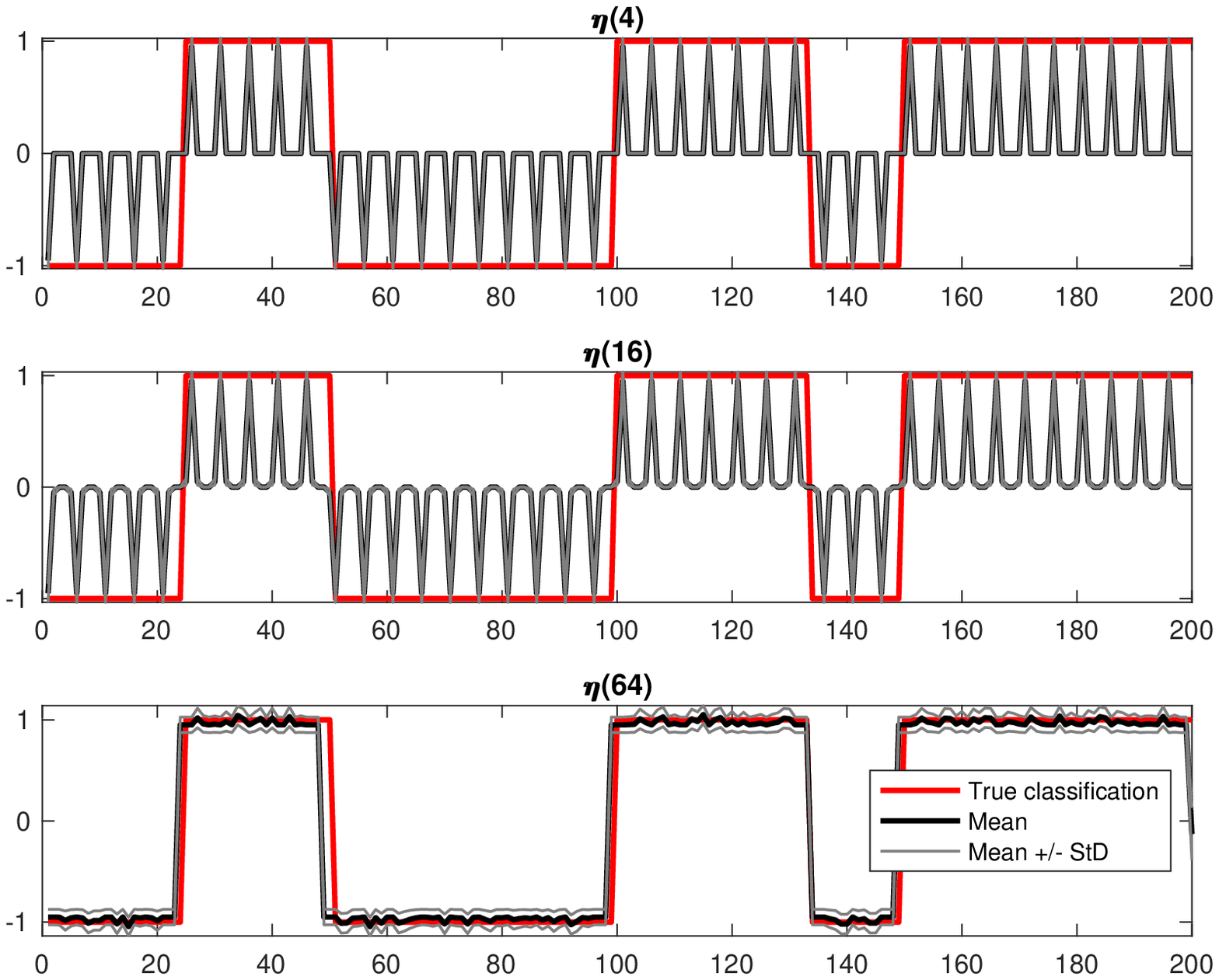}
    \caption{Plots of means (black thick lines), means $\pm$ standard deviations (grey thin lines)of $\eta(4), \eta(16), \eta(64)$, given $\varepsilon= 10^{-6}$, as well as $\lambda = 1$ (top) and $\lambda = 10$ (bottom). As a baseline, we compare with the true classification $\eta^\dagger$ (red).}
    \label{Fig:AC-6}
\end{figure}

None of the stochastic approximations is able to retrieve the precise classification. This is expected due to the sparsity of the data. The stochasticity of the algorithm does not seem to pick up these regions of uncertainty: the variance is \review{only insignificantly} increased around the jumps. Hence, we see also here that  the randomisation of what is essentially an MBO scheme \review{does not necessarily lead to a useful uncertainty quantification.}

\subsection{Classification in 2D} \label{Subsec_Ill_class2}
\review{Finally, we study a classification problem in 2D -- indeed, a two-dimensional version of the problem above. Here, we have index set $I:= \{1,\ldots,200\}^2$ and try to identify a classification vector $\eta^\dagger \in \{-1, 1\}^{200 \times 200}$.  We show a plot of $\eta^\dagger$ in Figure~\ref{fig:org_2D}. When classifying, we observe every $5^{\rm th}$ column in every $5^{\rm th}$ row, i.e. entries with indices in $\{(i,j): i,j \in \{1,6,\ldots,196 \}\}$ accumulating to $4\%$ of the domain.}
\begin{figure}
    \centering
    \includegraphics[scale = 0.5]{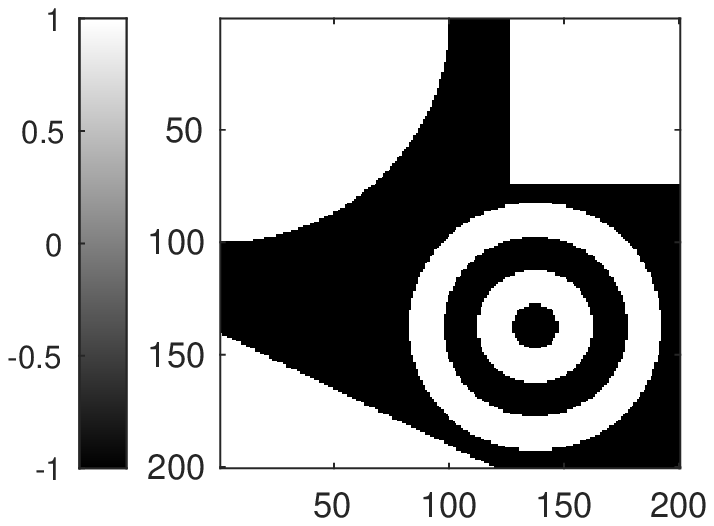}
    \caption{A plot of the underlying `true' classification $\eta^\dagger$ in the 2D example.}
    \label{fig:org_2D}
\end{figure}

\review{To employ the discrete Allen--Cahn equation, we construct a two-dimensional version of the one-dimensional Laplacian $\triangle'_1$ given in Subsection~\ref{Subsec_Ill_class1} by defining $\triangle_2' := \mathrm{Id}_X \otimes \triangle'_1 + \triangle'_1 \otimes \mathrm{Id}_{X}$. The discrete Allen--Cahn equation is now posed on the space $X := \mathbb{R}^{4\cdot 10^4}$. Solving the linear part of the problem \eqref{eq_AC_diff} by computing the matrix exponential at very high accuracy is hardly possible on such a high-dimensional space. Thus, we approximate a linear step to which we switch at time $T$ and that has a random step size $\delta T \sim \mathrm{Exp}(\lambda)$ by
$$
\left(\mathrm{Id}_{X}- \frac{\delta T}{2}(-\PtP+\varepsilon\triangle_2')\right)^{-1}\left(\xi(T-)+\frac{\delta T}{2}(-\PtP+\varepsilon\triangle_2')\xi(T-) + \delta T\Ptd\right).
$$
This corresponds to using a single step of the implicit midpoint rule or the Crank--Nicolson method \cite{crank_nicolson_1947} to integrate the ODE \eqref{eq_AC_diff}. See \cite{dubious} for other techniques to compute matrix exponentials and \cite{abdulle} for a discussion of ODE integrators with randomised step size. The thresholding step \eqref{eq_AC_thresh} is still evaluated accurately.}

\begin{figure}
    \centering
    \includegraphics[scale = 0.5]{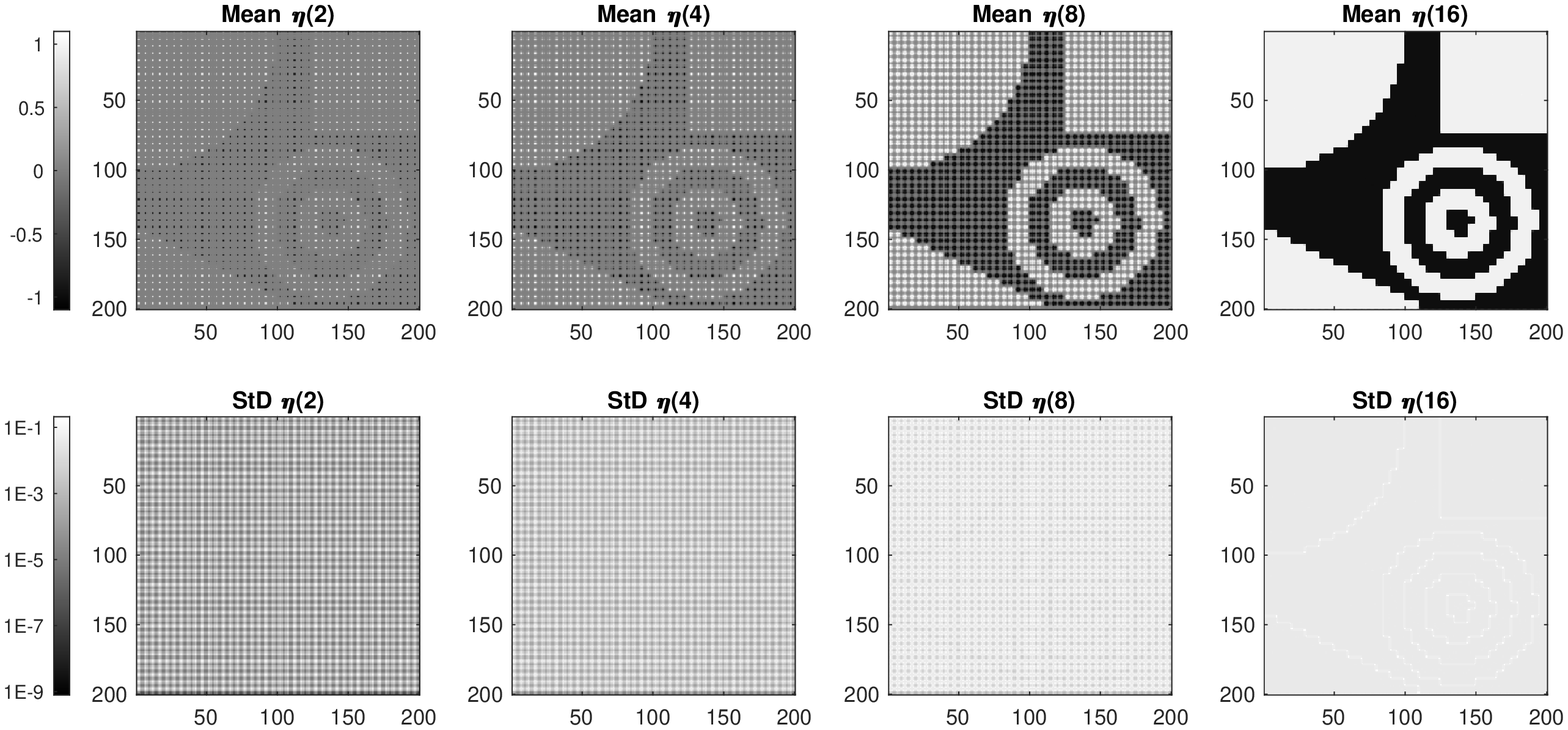}
     \caption{Plots of mean and standard deviations of the stochastic approximation of the discrete Allen--Cahn equation at times $t = 2, 4, 8, 16$ in the 2D experiment with $\varepsilon = 0.005$ and $\lambda = 20$. Means and standard deviations are computed over 100 experiments.}
    \label{fig:2D_epsi0_005}
\end{figure}

\begin{figure}
    \centering
    \includegraphics[scale = 0.5]{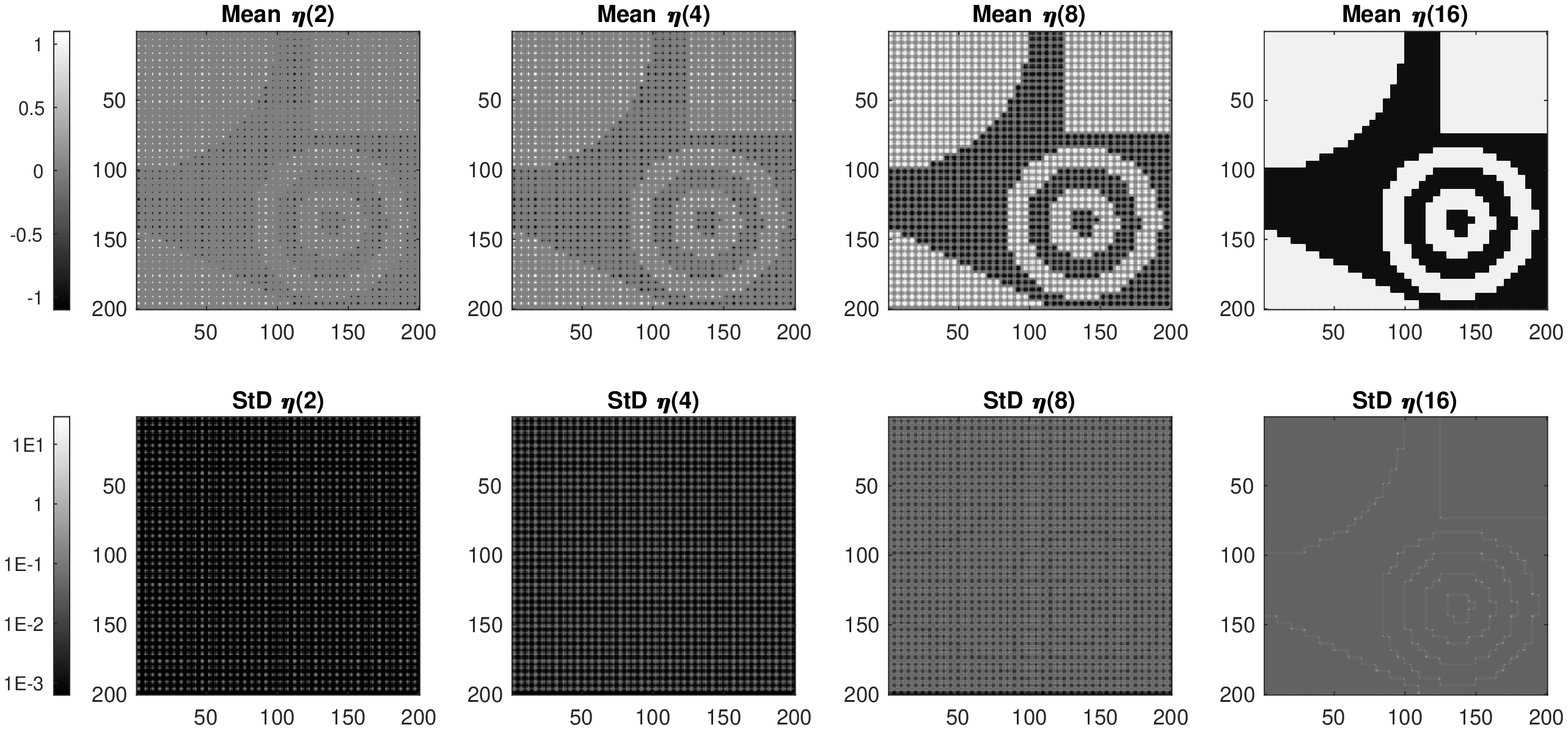}
    \caption{Plots of mean and standard deviations of the stochastic approximation of the discrete Allen--Cahn equation at times $t = 2, 4, 8, 16$ in the 2D experiment with $\varepsilon = 0.05$ and $\lambda = 20$. Means and standard deviations are computed over 100 experiments.}
    \label{fig:2D_epsi0_05}
\end{figure}

\review{We simulate the stochastic approximation for $\lambda = 20$, $\alpha := 1$,  $\varepsilon \in \{0.05, 0.005\}$, and initial condition $\eta(0) = 0$. We compute means and standard deviations  over 100 independent runs of $\eta(2), \eta(4), \eta(8), \eta(16)$. We present the results in Figures~\ref{fig:2D_epsi0_005} and \ref{fig:2D_epsi0_05}. Computing a realisation of $\eta(16)$ took on average $24.88(\pm 2.43)$ seconds.}

\review{The dynamical systems appear to behave similarly to the 1D non-discretised setting. In the beginning, the function fits the data at points of observation but remains 0 everywhere else. Over time, the function is slowly smoothed out leading to a good reconstruction of the true classification. In case $\varepsilon = 0.005$, we see all over smaller variances in the dynamical system. Moreover, at time  $t = 16$, we see an increment of the standard deviation around the interface in the reconstruction showing the uncertainty around the interface position. Unfortunately, this area is very small and does, e.g., not incorporate the truth. The numerical approximation of the linear part appears to not severely damage the dynamical system, as long as the waiting times are sufficiently small with high probability. Thus, a numerical approximation of the linear ODE appears to be a suitable technique, if $\lambda$ is sufficiently large.}

\section{Conclusions and outlook} \label{Sec_Conclusions}
In this work, we have proposed and analysed randomised, continuous-time splitting methods for sparse inversion and binary classification. The methods arise as stochastic approximations of the original dynamical systems, a sparse inversion flow, and a discrete Allen--Cahn equation. To obtain the stochastic approximation, we split the non-linear and the non-smooth part into two differential equations/inclusions. Each of these dynamical systems can be solved analytically or numerically \review{in an accurate and stable way}.  To approximate we alternately follow one of these subflows for an exponentially distributed waiting before switching. The non-smooth part leads in both cases to a randomised thresholding.

We give assumptions under which the stochastic approximations are exponentially ergodic, i.e., they converge quickly to their unique stationary measure. Moreover, we show that the stochastic approximations can indeed approximate the underlying dynamical systems at any accuracy. We finally show the applicability of our method in simple numerical experiments.

We finish this work by mentioning two possible extensions to the presented methods:

\paragraph{Data subsampling.} In case the data sets $b, d$ are very large, even a simple numerical approximation of the linear parts of the algorithms may be prohibitive. The traditional stochastic gradient descent method solves this problem by subsampling the data, i.e., partitioning the data vectors and optimising only with respect to one of the data subsets at a time. \review{The recent work} \cite{Latz2021} presents a natural way to introduce subsampling in\review{to} our stochastic approximations of the sparse inversion flow and the discrete Allen--Cahn equation.

\paragraph{Stochastic approximation of primal lower nice functions.} In this work, we focussed on two particular dynamical systems -- as opposed to the general settings we started with in Subsection~\ref{subsec_Problem_Setting}. Moreover, we have noticed that the theory of primal lower nice functions appears to be particularly suitable for our objectives. It would be interesting to see which of the statements made in this work still hold when replacing the given potentials by general primal lower nice functions. \review{Practically, this would allow to analyse stochastic splitting techniques for much more complicated non-smooth dynamical systems, such as the ROF flow and classification with other double-well potentials.}

\bibliography{library}

\begin{thebibliography}{10}

\bibitem{abdulle}
Assyr Abdulle and Giacomo Garegnani.
\newblock Random time step probabilistic methods for uncertainty quantification
  in chaotic and geometric numerical integration.
\newblock {\em Statistics and Computing}, 30(4):907--932, 2020.

\bibitem{Ahishakiye}
Emmanuel Ahishakiye, Martin Bastiaan~Van Gijzen, Julius Tuwiine, and Johnes
  Obungoloch.
\newblock A dictionary learning approach for noise-robust image reconstruction
  in low-field magnetic resonance imaging.
\newblock In {\em 2020 IST-Africa Conference (IST-Africa)}, pages 1--12, 2020.

\bibitem{AC_orig}
Samuel~M. Allen and John~W. Cahn.
\newblock {A microscopic theory for antiphase boundary motion and its
  application to antiphase domain coarsening}.
\newblock {\em Acta Metallurgica}, 27(6):1085--1095, 1979.

\bibitem{Atchade}
Yves~F. Atchad{{\'e}}, Gersende Fort, and Eric Moulines.
\newblock On perturbed proximal gradient algorithms.
\newblock {\em Journal of Machine Learning Research}, 18(10):1--33, 2017.

\bibitem{Benaim2012_quant}
Michel Benaïm, Stéphane~Le Borgne, Florent Malrieu, and Pierre-André Zitt.
\newblock {Quantitative ergodicity for some switched dynamical systems}.
\newblock {\em Electronic Communications in Probability}, 17:1 -- 14, 2012.

\bibitem{BENES2004}
Michal Bene\v{s}, Vladim\'ir Chalupeck\'y, and Karol Mikula.
\newblock {Geometrical image segmentation by the Allen--Cahn equation}.
\newblock {\em Applied Numerical Mathematics}, 51(2):187--205, 2004.

\bibitem{BERTACCO2021112122}
Federico Bertacco.
\newblock {Stochastic Allen--Cahn equation with logarithmic potential}.
\newblock {\em Nonlinear Analysis}, 202:112122, 2021.

\bibitem{Bertozzi-Flenner}
Andrea~L. Bertozzi and Arjuna Flenner.
\newblock {Diffuse Interface Models on Graphs for Classification of High
  Dimensional Data}.
\newblock {\em Multiscale Modeling \& Simulation}, 10(3):1090--1118, 2012.

\bibitem{Bertozzi}
{Bertozzi, Andrea L. and Luo, Xiyang and Stuart, Andrew M. and Zygalakis,
  Konstantinos C.}
\newblock Uncertainty quantification in graph-based classification of high
  dimensional data.
\newblock {\em SIAM/ASA Journal on Uncertainty Quantification}, 6(2):568--595,
  2018.

\bibitem{Bierkens}
Joris Bierkens, Paul Fearnhead, and Gareth Roberts.
\newblock {The Zig-Zag process and super-efficient sampling for Bayesian
  analysis of big data}.
\newblock {\em The Annals of Statistics}, 47(3):1288 -- 1320, 2019.

\bibitem{Brezis}
Haim Br\'ezis.
\newblock {\em {Operateurs Maximaux Monotones}}.
\newblock North-Holland Mathematics Studies. North-Holland, 1973.

\bibitem{BRUCK1975}
Ronald~E Bruck.
\newblock {Asymptotic convergence of nonlinear contraction semigroups in
  Hilbert space}.
\newblock {\em Journal of Functional Analysis}, 18(1):15--26, 1975.

\bibitem{BuddVG}
Jeremy Budd and Yves~Van Gennip.
\newblock {Graph Merriman--Bence--Osher as a SemiDiscrete Implicit Euler Scheme
  for Graph Allen--Cahn Flow}.
\newblock {\em SIAM Journal on Mathematical Analysis}, 52(5):4101--4139, 2020.

\bibitem{BuddvGL}
Jeremy Budd, Yves van Gennip, and Jonas Latz.
\newblock {Classification and image processing with a semi-discrete scheme for
  fidelity forced Allen--Cahn on graphs}.
\newblock {\em GAMM-Mitteilungen}, 44(1):e202100004, 2021.

\bibitem{Bungert_2019}
Leon Bungert and Martin Burger.
\newblock {Solution paths of variational regularization methods for inverse
  problems}.
\newblock {\em Inverse Problems}, 35(10):105012, sep 2019.

\bibitem{CANDES2011}
Emmanuel~J. Candès, Yonina~C. Eldar, Deanna Needell, and Paige Randall.
\newblock {Compressed sensing with coherent and redundant dictionaries}.
\newblock {\em Applied and Computational Harmonic Analysis}, 31(1):59--73,
  2011.

\bibitem{Cloez2015}
Bertrand Cloez and Martin Hairer.
\newblock {Exponential ergodicity for Markov processes with random switching}.
\newblock {\em Bernoulli}, 21(1):505 -- 536, 2015.

\bibitem{Combettes2011}
Patrick~L. Combettes and Jean-Christophe Pesquet.
\newblock {Proximal Splitting Methods in Signal Processing}.
\newblock In Heinz~H. Bauschke, Regina~S. Burachik, Patrick~L. Combettes, Veit
  Elser, D.~Russell Luke, and Henry Wolkowicz, editors, {\em {Fixed-Point
  Algorithms for Inverse Problems in Science and Engineering}}, pages 185--212.
  Springer New York, New York, NY, 2011.

\bibitem{crank_nicolson_1947}
J.~Crank and P.~Nicolson.
\newblock A practical method for numerical evaluation of solutions of partial
  differential equations of the heat-conduction type.
\newblock {\em Mathematical Proceedings of the Cambridge Philosophical
  Society}, 43(1):50–67, 1947.

\bibitem{Cucuringu2021}
Mihai Cucuringu, Andrea Pizzoferrato, and Yves van Gennip.
\newblock {An MBO scheme for clustering and semi-supervised clustering of
  signed networks}.
\newblock {\em Communications in Mathematical Sciences}, 19(1):73 -- 109, 2021.

\bibitem{Davis1984}
M.~H.~A. Davis.
\newblock {Piecewise-Deterministic Markov Processes: A General Class of
  Non-Diffusion Stochastic Models}.
\newblock {\em Journal of the Royal Statistical Society. Series B
  (Methodological)}, 46(3):353--388, 1984.

\bibitem{Dupuis}
Paul Dupuis, Yufei Liu, Nuria Plattner, and J.~D. Doll.
\newblock {On the Infinite Swapping Limit for Parallel Tempering}.
\newblock {\em Multiscale Modeling \& Simulation}, 10(3):986--1022, 2012.

\bibitem{Esedoglu06}
Selim Esedoglu and Yen-Hsi~Richard Tsai.
\newblock {Threshold dynamics for the piecewise constant Mumford--Shah
  functional}.
\newblock {\em Journal of Computational Physics}, 211(1):367--384, 2006.

\bibitem{Feng}
Xiaobing Feng and Andreas Prohl.
\newblock {Numerical analysis of the Allen-Cahn equation and approximation for
  mean curvature flows}.
\newblock {\em Numerische Mathematik}, 94(1):33--65, 2003.

\bibitem{Goldstein}
Tom Goldstein, Christoph Studer, and Richard Baraniuk.
\newblock {A Field Guide to Forward-Backward Splitting with a FASTA
  Implementation}.
\newblock {\em arXiv eprint}, abs/1411.3406, 2014.

\bibitem{Jin2021}
Kexin {Jin}, Jonas {Latz}, Chenguang {Liu}, and Carola-Bibiane {Sch{\"o}nlieb}.
\newblock {A Continuous-time Stochastic Gradient Descent Method for Continuous
  Data}.
\newblock {\em arXiv e-prints}, page arXiv:2112.03754, December 2021.

\bibitem{Kallenberg}
Olav Kallenberg.
\newblock {\em "Foundations of Modern Probability "}.
\newblock Springer International Publishing, Cham, 2021.

\bibitem{Kushner1984}
Harold~J. Kushner.
\newblock {\em {A}pproximation and {W}eak {C}onvergence {M}ethods for {R}andom
  {P}rocesses, with {A}pplications to {S}tochastic {S}ystems {T}heory}.
\newblock The MIT Press, 1984.

\bibitem{Kushner1990}
Harold~J. Kushner.
\newblock {\em {W}eak {C}onvergence {M}ethods and {S}ingularly {P}erturbed
  {S}tochastic {C}ontrol and {F}iltering {P}roblems}.
\newblock Birkh\"auser, 1990.

\bibitem{Latz2021}
Jonas Latz.
\newblock {Analysis of stochastic gradient descent in continuous time}.
\newblock {\em Statistics and Computing}, 31:39, 2021.

\bibitem{Lee2022}
Donghun Lee, Sangkwon Kim, Hyun~Geun Lee, Soobin Kwak, Jian Wang, and Junseok
  Kim.
\newblock {Classification of ternary data using the ternary Allen--Cahn system
  for small datasets}.
\newblock {\em AIP Advances}, 12(6):065324, 2022.

\bibitem{LI20101591}
Yibao Li, Hyun~Geun Lee, Darae Jeong, and Junseok Kim.
\newblock {An unconditionally stable hybrid numerical method for solving the
  Allen--Cahn equation}.
\newblock {\em Computers \& Mathematics with Applications}, 60(6):1591--1606,
  2010.

\bibitem{Mandt2017}
Stephan Mandt, Matthew~D. Hoffman, and David~M. Blei.
\newblock {Stochastic Gradient Descent as Approximate Bayesian Inference}.
\newblock {\em Journal of Machine Learning Research}, 18(134):1--35, 2017.

\bibitem{Marcellin2006}
Sylvie Marcellin and Lionel Thibault.
\newblock {Evolution Problems Associated with Primal Lower Nice Functions}.
\newblock {\em Journal of Convex Analysis}, 13:385--421, 2006.

\bibitem{MBO}
Barry Merriman, James Bence, and Stanley Osher.
\newblock Diffusion generated motion by mean curvature.
\newblock {\em UCLA CAM Report 92-18}, 1992.

\bibitem{MIN20095256}
Jae~H. Min and Chulwoo Jeong.
\newblock {A binary classification method for bankruptcy prediction}.
\newblock {\em Expert Systems with Applications}, 36(3, Part 1):5256--5263,
  2009.

\bibitem{pmlr-v162-mishchenko22b}
Konstantin Mishchenko, Grigory Malinovsky, Sebastian Stich, and Peter
  Richtarik.
\newblock {P}rox{S}kip: Yes! {L}ocal gradient steps provably lead to
  communication acceleration! {F}inally!
\newblock In Kamalika Chaudhuri, Stefanie Jegelka, Le~Song, Csaba Szepesvari,
  Gang Niu, and Sivan Sabato, editors, {\em Proceedings of the 39th
  International Conference on Machine Learning}, volume 162 of {\em Proceedings
  of Machine Learning Research}, pages 15750--15769. PMLR, 17--23 Jul 2022.

\bibitem{dubious}
Cleve Moler and Charles~Van Loan.
\newblock {Nineteen Dubious Ways to Compute the Exponential of a Matrix,
  Twenty-Five Years Later}.
\newblock {\em SIAM Review}, 45(1):3--49, 2003.

\bibitem{RobbinsMonro}
Herbert Robbins and Sutton Monro.
\newblock {A Stochastic Approximation Method}.
\newblock {\em The Annals of Mathematical Statistics}, 22(3):400 -- 407, 1951.

\bibitem{Rosasco}
Lorenzo Rosasco, Silvia Villa, and Bang~C{\^o}ng V{\~u}.
\newblock Stochastic forward--backward splitting for monotone inclusions.
\newblock {\em Journal of Optimization Theory and Applications},
  169(2):388--406, 2016.

\bibitem{Saner1997}
Halil~Mete Saner.
\newblock {Ginzburg--Landau equation and motion by mean curvature, I:
  Convergence}.
\newblock {\em The Journal of Geometric Analysis}, 7(3):437--475, Sep 1997.

\bibitem{Scholtes2012}
Stefan Scholtes.
\newblock {\em {Introduction to Piecewise Differentiable Equations}}.
\newblock Springer New York, New York, NY, 2012.

\bibitem{Shor1985}
Naun~Zuselevich Shor.
\newblock {\em {Minimization Methods for Non-Differentiable Functions}}.
\newblock Springer Berlin Heidelberg, 1985.

\bibitem{smith2021on}
Samuel~L Smith, Benoit Dherin, David Barrett, and Soham De.
\newblock {On the Origin of Implicit Regularization in Stochastic Gradient
  Descent}.
\newblock In {\em ICLR}, 2021.

\bibitem{Stadler2007}
Georg Stadler.
\newblock {Elliptic optimal control problems with L1-control cost and
  applications for the placement of control devices}.
\newblock {\em Computational Optimization and Applications}, 44(2):159, Nov
  2007.

\bibitem{Guoshao}
Guoshao Su, Yan Zhang, and Guoqing Chen.
\newblock {Identify rockburst Grades for Jinping II hydropower station using
  Gaussian Process for Binary Classification}.
\newblock In {\em 2010 International Conference on Computer, Mechatronics,
  Control and Electronic Engineering}, volume~2, pages 364--367, 2010.

\bibitem{Ting_2009}
F~Ting.
\newblock {Effective dynamics of multi-vortices in an external potential for
  the Ginzburg--Landau gradient flow}.
\newblock {\em Nonlinearity}, 23(1):179--210, dec 2009.

\bibitem{Tufail2020-ee}
Ahsan~Bin Tufail, Yong-Kui Ma, and Qiu-Na Zhang.
\newblock {Binary Classification of Alzheimer's Disease Using {sMRI} Imaging
  Modality and Deep Learning}.
\newblock {\em J Digit Imaging}, 33(5):1073--1090, October 2020.

\bibitem{Villani}
C\'edric Villani.
\newblock {\em {O}ptimal {T}ransport}.
\newblock Springer, 2009.

\bibitem{Yang2016-wp}
Alice~C Yang, Madison Kretzler, Sonja Sudarski, Vikas Gulani, and Nicole
  Seiberlich.
\newblock {Sparse Reconstruction Techniques in Magnetic Resonance Imaging:
  Methods, Applications, and Challenges to Clinical Adoption}.
\newblock {\em Invest Radiol}, 51(6):349--364, June 2016.

\bibitem{Zhao}
Rongchang Zhao, Hong Li, and Xiyao Liu.
\newblock {A Survey of Dictionary Learning in Medical Image Analysis and Its
  Application for Glaucoma Diagnosis}.
\newblock {\em Archives of Computational Methods in Engineering},
  28(2):463--471, 2021.

\end{thebibliography}
\bibliographystyle{plain}
\end{document}